\documentclass[12pt,a4paper,dvipsnames]{amsart}
\usepackage[utf8]{inputenc}
\usepackage{float}
\usepackage{caption}
\usepackage{subcaption}
\usepackage{tabularx}
\usepackage[table]{xcolor}
\usepackage{amsmath, nicefrac, amsthm, verbatim, amsfonts, amssymb, xcolor, bbm}
\usepackage{graphics, xspace, enumerate}
\usepackage[a4paper,margin=3cm]{geometry}

\usepackage[bb=boondox]{mathalfa}
\usepackage{tikz}
\usepackage{graphicx}
\usepackage[colorlinks=true,citecolor=green,urlcolor=green,linkcolor=green,bookmarksopen=true,unicode=true,pdffitwindow=true]{hyperref}
\usepackage[english]{babel}
\usepackage[languagenames,fixlanguage]{babelbib}

\usepackage[labelformat=simple]{subcaption}

\makeatletter
\@namedef{subjclassname@2020}{\textup{2020} Mathematics Subject Classification}
\makeatother

\theoremstyle{plain}
\newtheorem{theorem}{Theorem}[section]

\newtheorem{lemma}[theorem]{Lemma}
\newtheorem{proposition}[theorem]{Proposition}
\theoremstyle{definition}
\newtheorem{remark}[theorem]{Remark}

\newtheorem{definition}[theorem]{Definition}

\newtheorem{question}[theorem]{Question}

\newcommand{\bbN}{\mathbb{N}}
\newcommand{\bbZ}{\mathbb{Z}}
\newcommand{\bbR}{\mathbb{R}}

\newcommand{\floor}[1]{\lfloor #1 \rfloor}
\newcommand{\FLOOR}[1]{\left\lfloor #1 \right\rfloor}

\newcommand{\CEIL}[1]{\left\lceil #1 \right\rceil}

\newcommand{\FJADEF}[3]{{#1}:{#2}\to{#3}}

\newcommand{\Mod}[1]{\ (\mathrm{mod}\ #1)}

\numberwithin{equation}{section}

\definecolor{Maroon}{RGB}{140,10,0}

\hypersetup{
	colorlinks,
	linkcolor={Maroon},
	citecolor={Maroon},
	urlcolor={Maroon}
}

\title{Combinatorial settlement planning}

\author[M.\ Puljiz]{Mate\ Puljiz}
\address[Mate Puljiz]{Department of Applied Mathematics\\
	Faculty of Electrical Engineering and Computing\\
	University of Zagreb\\ 
 Zagreb\\ 
	Croatia}
\email{mate.puljiz@fer.hr}

\author[S.\ \v{S}ebek]{Stjepan\ \v{S}ebek}
\address[Stjepan\ \v{S}ebek]{Department of Applied Mathematics\\
	Faculty of Electrical Engineering and Computing\\
	University of Zagreb\\ 
 Zagreb\\ 
	Croatia}
\email{stjepan.sebek@fer.hr}

\author[J.\ \v{Z}ubrini\'{c}]{Josip\ \v{Z}ubrini\'{c}}
\address[Josip\ \v{Z}ubrini\'{c}]{Department of Applied Mathematics\\
	Faculty of Electrical Engineering and Computing\\
	University of Zagreb\\ 
 Zagreb\\ 
	Croatia}
\email{josip.zubrinic@fer.hr}


\begin{document}

\begin{abstract}
	In this article, we consider a combinatorial settlement model on a rectangular grid where at least one side (east, south or west) of each house must be exposed to sunlight without obstructions. We are interested in maximal configurations, where no additional houses can be added. For a fixed $m\times n$ grid we explicitly calculate the lowest number of houses, and give close to optimal bounds on the highest number of houses that a maximal configuration can have. Additionally, we provide an integer programming formulation of the problem and solve it explicitly for small values of $m$ and $n$.
\end{abstract}

\maketitle

%
%
%
%

\section{Introduction}\label{sec:intro}

Consider the following problem: a rectangular $m\times n$ tract of land, whose sides are oriented north-south and east-west as in Figure \ref{fig:tract_of_land}, consists of $mn$ square lots of size $1 \times 1$. Each $1 \times 1$ square lot can be either empty, or occupied by a single house. A house is said to be \emph{blocked from sunlight} if the three lots immediately to its east, west and south are all occupied (it is assumed that sunlight always comes from the south\footnote{Our Southern Hemisphere friends are welcome to turn the page upside down when inspecting the figures in our paper.}). Along the boundary of the rectangular $m \times n$ grid, there are no obstructions to sunlight. We refer to the models of such rectangular tracts of land, with certain lots occupied, as configurations. Of interest are maximal configurations, where no house is blocked from the sunlight, and any further addition of a house to the configuration on any empty lot would result in either that house being blocked from the sunlight, or it would cut off sunlight from some previously built house, or both.

\begin{figure}[h]
	\centering
	\begin{tikzpicture}[scale = 0.5]
	\draw[step=1cm,black,very thin] (0, 0) grid (7,5);
	\draw [->,>=stealth] (9,1) -- (9,4);
	\node[anchor=west] at (9,4) {North};
	\end{tikzpicture}
	\caption{An example of a tract of land ($m = 5$, $n = 7$).}
	\label{fig:tract_of_land}
\end{figure}
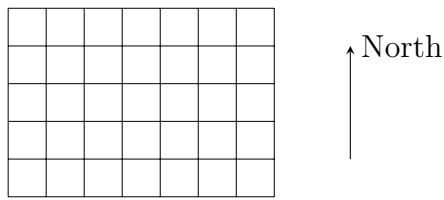

\medskip

We can encode any fixed configuration as a 0-1, $m\times n$ matrix $C$, with $C_{i,j}=1$ if and only if a house is built on the lot $(i,j)$ ($i$-th row and $j$-th column, counted from the top left corner). We can, equivalently, think of $C$ as a subset of $[m]\times[n] = \{(i,j) : 1\le i \le m,\, 1\le j\le n\}$, where, again, $(i,j)\in C$ if and only if a house is built on the lot $(i,j)$.

It is natural to define \emph{building density} of a configuration $C$ as $\dfrac{|C|}{mn}$, where 
$$|C| = \sum_{i=1}^m\sum_{j=1}^n C_{i,j}$$
is the total number of occupied lots in the configuration $C$, i.e.\ the cardinality of $C$ when $C$ is interpreted as a subset of $[m]\times [n]$. We also call $|C|$ the \emph{occupancy} of $C$.

A configuration $C$ is said to be \emph{permissible} if no house in it is blocked from the sunlight, otherwise it is called \emph{impermissible}.

A configuration $C$ is said to be \emph{maximal} if it is permissible and no other permissible configuration strictly contains it, i.e.\ no further houses can be added to it, whilst ensuring that all the houses still get some sunlight. See Figure \ref{fig:examples} for examples of impermissible, permissible and maximal configuration on a $5 \times 4$ tract of land. Shaded squares represent houses and unshaded squares represent empty lots on the tract of land. The houses that are blocked from the sunlight are marked with the letter x.

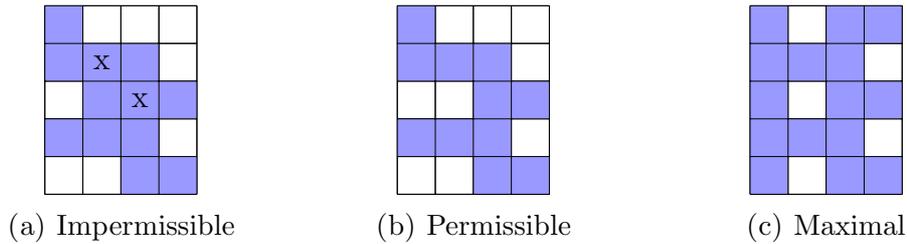
\begin{figure}
	\begin{subfigure}{0.3\textwidth}\centering
		\begin{tikzpicture}[scale = 0.5]
		    \draw[step=1cm,black,very thin] (0, 0) grid (4,5);
		    \fill[blue!40!white] (0,1) rectangle (1,2);
		    \fill[blue!40!white] (0,3) rectangle (1,4);
		    \fill[blue!40!white] (0,4) rectangle (1,5);
		    \fill[blue!40!white] (1,1) rectangle (2,2);
		    \fill[blue!40!white] (1,2) rectangle (2,3);
		    \fill[blue!40!white] (1,3) rectangle (2,4);
		    \node[] at (1.5,3.5) {x};
		    \fill[blue!40!white] (2,0) rectangle (3,1);
		    \fill[blue!40!white] (2,1) rectangle (3,2);
		    \fill[blue!40!white] (2,2) rectangle (3,3);
		    \node[] at (2.5,2.5) {x};
		    \fill[blue!40!white] (2,3) rectangle (3,4);
		    \fill[blue!40!white] (3,0) rectangle (4,1);
		    \fill[blue!40!white] (3,2) rectangle (4,3);
		    \draw[step=1cm,black,very thin] (0, 0) grid (4,5);
		\end{tikzpicture}
		\caption{Impermissible}
	\end{subfigure}
	\begin{subfigure}{0.3\textwidth}\centering
		\begin{tikzpicture}[scale = 0.5]
		    \draw[step=1cm,black,very thin] (0, 0) grid (4,5);
		    \fill[blue!40!white] (0,1) rectangle (1,2);
		    \fill[blue!40!white] (0,3) rectangle (1,4);
		    \fill[blue!40!white] (0,4) rectangle (1,5);
		    \fill[blue!40!white] (1,1) rectangle (2,2);
		    \fill[blue!40!white] (1,3) rectangle (2,4);
		    \fill[blue!40!white] (2,0) rectangle (3,1);
		    \fill[blue!40!white] (2,1) rectangle (3,2);
		    \fill[blue!40!white] (2,2) rectangle (3,3);
		    \fill[blue!40!white] (2,3) rectangle (3,4);
		    \fill[blue!40!white] (3,0) rectangle (4,1);
		    \fill[blue!40!white] (3,2) rectangle (4,3);
		    \draw[step=1cm,black,very thin] (0, 0) grid (4,5);
		\end{tikzpicture}
		\caption{Permissible}
	\end{subfigure}
	\begin{subfigure}{0.3\textwidth}\centering
		\begin{tikzpicture}[scale = 0.5]
		    \draw[step=1cm,black,very thin] (0, 0) grid (4,5);
		    \fill[blue!40!white] (0,0) rectangle (1,1);
		    \fill[blue!40!white] (0,1) rectangle (1,2);
		    \fill[blue!40!white] (0,2) rectangle (1,3);
		    \fill[blue!40!white] (0,3) rectangle (1,4);
		    \fill[blue!40!white] (0,4) rectangle (1,5);
		    \fill[blue!40!white] (1,1) rectangle (2,2);
		    \fill[blue!40!white] (1,3) rectangle (2,4);
		    \fill[blue!40!white] (2,0) rectangle (3,1);
		    \fill[blue!40!white] (2,1) rectangle (3,2);
		    \fill[blue!40!white] (2,2) rectangle (3,3);
		    \fill[blue!40!white] (2,3) rectangle (3,4);
		    \fill[blue!40!white] (2,4) rectangle (3,5);
		    \fill[blue!40!white] (3,0) rectangle (4,1);
		    \fill[blue!40!white] (3,2) rectangle (4,3);
		    \fill[blue!40!white] (3,4) rectangle (4,5);
		    \draw[step=1cm,black,very thin] (0, 0) grid (4,5);
		\end{tikzpicture}
		\caption{Maximal}
	\end{subfigure}
	\caption{Examples of impermissible, permissible and maximal configuration on a $5 \times 4$ tract of land.}\label{fig:examples}
\end{figure}
One is naturally interested in maximal configurations, especially those that achieve the highest and the lowest building density, or occupancy. One can think of those as solving one of two natural optimization problems:
\begin{enumerate}[Problem 1.]\it
	\item Maximize the revenue of a real estate investor, by building as many houses as possible, whilst ensuring that each house gets some sunlight.
	
	\item Maximize the quality of living, by arranging the houses in order to achieve the lowest building density possible, whilst ensuring that no new houses can be added in future without compromising access to sunlight.
\end{enumerate}

Maximal configurations, for a particular $m,n\in\bbN$, with the highest building density possible  we call \emph{efficient}, while those with the lowest building density possible we call \emph{inefficient}. We denote the number of occupied lots (occupancy) in any of the efficient configurations by $E_{m,n}$, and the number of occupied lots in any of the inefficient configurations by $I_{m,n}$.

One may also be interested in maximal configurations which exhibit occupancies in between those two extremes. We study these and related problems in a forthcoming paper \cite{PSZ-21-2}.

We were introduced to this problem by Juraj Bo\v{z}i\'{c} who came up with it during his studies at Faculty of Architecture, University of Zagreb. His main goal was to design a model for settlement planning where the impact of the architect would be as small as possible and people would have a lot of freedom in the process of building the settlement. This minimal intervention from the side of the architect is given through the condition that houses are not allowed to be blocked from the sunlight and that the tracts of land on which the settlements are built are of rectangular shapes.

\begin{remark}[Boundary condition]
	As stated in the formulation of the problem, there are no obstructions to sunlight along the boundary of the rectangular $m\times n$ grid. However, it is possible to consider various other boundary conditions. One might be interested in the case where the whole border is bricked up and houses can get sunlight only from the empty lots within the grid. In fact, it is not hard to show that for each $m,n\ge 2$ there is an efficient configuration with all lots along the southern, eastern and western border occupied. This implies $E_{m,n} = \check{E}_{m-1,n-2}+2m+n-2$ where $\check{E}_{m-1,n-2}$ is the number of occupied lots in any efficient $(m-1)\times (n-2)$ configuration with the bricked up boundary. Note that the boundary condition along the northern side of the border is irrelevant.
\end{remark}

The rest of the paper is organized as follows. In Section \ref{sec:simple_bounds} we provide coarse bounds on the size of maximal configurations.  In Section \ref{sec:nearoptimal} we introduce periodic configurations on $\bbZ^2$ which, when restricted to a finite grid, come close to attaining the occupancy of efficient or inefficient configurations. In Section \ref{sec:in_eficientformula} we find a closed formula for $I_{m,n}$ and improve the bounds on $E_{m,n}$. In Section \ref{sec:ip} we give an integer programming formulation of our optimization problems and compute $E_{m,n}$ explicitly, for values $m,n\le 16$, using IBM ILOG CPLEX \cite{cplex2009v12}. Finally, in Section \ref{sec:alt_formulations}, we provide alternative formulations of the problem of finding efficient configurations.

\section{Two simple bounds and asymptotics for large \texorpdfstring{$m$}{m} and \texorpdfstring{$n$}{n}}\label{sec:simple_bounds}

\begin{lemma}\label{lm:crudeBound}
	If $C$ is any maximal configuration on the $m\times n$ grid, $m,n\ge 2$, then $$\frac{1}{2}mn\le|C|\le \frac{3}{4}mn+\frac{m-1}{2}+\frac{n}{4}.$$
\end{lemma}
\begin{proof}
	We first prove the upper bound. If we interpret $C$ as the set of all occupied lots, then its complement $C^c$ is the set of all empty lots. Note that each occupied lot is either on the eastern, southern or western edge of the grid (there are $2m+n-2$ such lots); or it gets sunlight from at least one empty lot in $C^c$; or possibly both. Furthermore, each empty lot in $C^c$ gives light to at most 3 occupied lots in $C$. Therefore,
	$$|C| \le 3|C^c|+2m+n-2,$$
	$$4|C| \le3(|C^c|+|C|)+2m+n-2=3mn+2m+n-2,$$
	and the upper bound follows.
	
	To prove the lower bound, it suffices to construct an injection from $C^c$ to $C$, as then, the bound would follow from
	$$mn=|C|+|C^c|\le 2|C|.$$
	To that end, consider any empty lot $(i,j)\in C^c$. Since $C$ is maximal, the house cannot be built on the lot $(i,j)$, i.e.\ $(i,j)$ cannot be added to the configuration, without blocking some existing house (or itself) from the sunlight. This can happen in exactly two ways:
	\begin{itemize}
		\item either $(i,j)$ is the only source of light to at least one of its occupied neighbors to the east, north or west: $(i,j+1)$, $(i-1,j)$, or $(i,j-1)$; (note that we do not require that all of those three neighbors are occupied)
		\item or alternatively, its neighbors to the east, south and west: $(i,j+1)$, $(i+1,j)$, and $(i,j-1)$ are all occupied (but $(i,j)$ is not the only source of light to $(i,j+1)$, nor $(i-1,j)$, nor $(i,j-1)$, if occupied).
	\end{itemize}
	In the first case, we map $(i,j)\in C^c$ to any of the neighbors: $(i,j+1)$, $(i-1,j)$, or $(i,j-1)$ that are occupied and for which $(i,j)$ is the only source of light.
	In the second case, we map $(i,j)\in C^c$ to its east neighbor $(i,j+1)$. It is clear from the construction that this does define an injection from $C^c$ to $C$, thus completing the proof of the lower bound.
\end{proof}

\begin{remark}
	Lemma \ref{lm:crudeBound} shows that for large grids, as both $m\to\infty$ and $n\to\infty$, the building density of maximal configurations must be between $\frac{1}{2}$ and $\frac{3}{4}$. We will later see that efficient configurations, in the limit, do approach building density $\frac{3}{4}$, while inefficient configurations, in the limit, do approach building density $\frac{1}{2}$. It is also interesting to note that these limit building densities are in, in fact, attained on the infinite grid by periodic configurations introduced in Section \ref{sec:nearoptimal}.
\end{remark}

\section{(Near-)optimal patterns}\label{sec:nearoptimal}
In this section we describe periodic configurations on $\mathbb{Z}^2$ which yield optimal building densities, and the associated finite versions of these configurations which obey a similar periodic rule. The motivation for this is to construct somewhat regular configurations with near-optimal building densities. For this, we introduce the following definitions: 

A configuration on $\mathbb{Z}^2$ is simply a subset of $\mathbb{Z}^2$, or equivalently a function belonging to $\{0,1\}^{\mathbb{Z}^2}$. For a fixed configuration $C$ on $\mathbb{Z}^2$, we define its building density as: 
\begin{equation*}
    \mathcal{D}(C) = \lim_{n \to \infty} \frac{1}{(2n+1)^2}\sum_{i=-n}^n \sum_{j = -n}^n C_{i,j}
\end{equation*}
when this limit exists.

A similar argument as in Lemma \ref{lm:crudeBound} yields the following bounds for the building density of maximal configurations on $\mathbb{Z}^2$: 
\begin{equation*}
    \frac{1}{2} \leq \mathcal{D}(C) \leq \frac{3}{4}, \quad \forall C\in \{0,1\}^{\mathbb{Z}^2}, \mbox{ $C$ maximal. } 
\end{equation*}

\subsection{Patterns occurring in efficient configurations}
\subsubsection{Brick pattern}\label{subsubsect:brickPattern}
Here we introduce the pattern for maximal configurations which we refer to as the brick pattern, see Figure \ref{fig:brickfigure}. This pattern defines a maximal configuration $C^{\rm brick}$ on $\mathbb{Z}^2$ with the highest possible building density of $\mathcal{D}(C^{\rm brick}) = 3/4$.
\begin{figure}
        \centering
        \begin{tikzpicture}[scale = 0.5]
            \draw[step=1cm,black, thin] (-0.9, -0.9) grid (11.9,4.9);
            \fill[blue!40!white] (-0.9,-0.9) rectangle (0,0);
            \fill[blue!40!white] (0,-0.9) rectangle (1,0);

            \fill[blue!40!white] (2,-0.9) rectangle (3,0);
            \fill[blue!40!white] (3,-0.9) rectangle (4,0);
            \fill[blue!40!white] (4,-0.9) rectangle (5,0);
  
            \fill[blue!40!white] (6,-0.9) rectangle (7,0);
            \fill[blue!40!white] (7,-0.9) rectangle (8,0);
            \fill[blue!40!white] (8,-0.9) rectangle (9,0);
            
            \fill[blue!40!white] (10,-0.9) rectangle (11,0);
            \fill[blue!40!white] (11,-0.9) rectangle (11.9,0);
            \fill[blue!40!white] (0,0) rectangle (1,1);
            \fill[blue!40!white] (1,0) rectangle (2,1);
            \fill[blue!40!white] (2,0) rectangle (3,1);

            \fill[blue!40!white] (4,0) rectangle (5,1);
            \fill[blue!40!white] (5,0) rectangle (6,1);
            \fill[blue!40!white] (6,0) rectangle (7,1);
            
            \fill[blue!40!white] (8,0) rectangle (9,1);
            \fill[blue!40!white] (9,0) rectangle (10,1);        \fill[blue!40!white] (10,0) rectangle (11,1);
            \fill[blue!40!white] (-0.9,1) rectangle (0,2);
            \fill[blue!40!white] (0,1) rectangle (1,2);

            \fill[blue!40!white] (2,1) rectangle (3,2);
            \fill[blue!40!white] (3,1) rectangle (4,2);
            \fill[blue!40!white] (4,1) rectangle (5,2);
  
            \fill[blue!40!white] (6,1) rectangle (7,2);
            \fill[blue!40!white] (7,1) rectangle (8,2);
            \fill[blue!40!white] (8,1) rectangle (9,2);
            
            \fill[blue!40!white] (10,1) rectangle (11,2);
            \fill[blue!40!white] (11,1) rectangle (11.9,2);
            \fill[blue!40!white] (0,2) rectangle (1,3);
            \fill[blue!40!white] (1,2) rectangle (2,3);
            \fill[blue!40!white] (2,2) rectangle (3,3);

            \fill[blue!40!white] (4,2) rectangle (5,3);
            \fill[blue!40!white] (5,2) rectangle (6,3);
            \fill[blue!40!white] (6,2) rectangle (7,3);
            
            \fill[blue!40!white] (8,2) rectangle (9,3);
            \fill[blue!40!white] (9,2) rectangle (10,3);        \fill[blue!40!white] (10,2) rectangle (11,3);
                    \fill[blue!40!white] (-0.9,3) rectangle (0,4);
            \fill[blue!40!white] (0,3) rectangle (1,4);

            \fill[blue!40!white] (2,3) rectangle (3,4);
            \fill[blue!40!white] (3,3) rectangle (4,4);
            \fill[blue!40!white] (4,3) rectangle (5,4);
  
            \fill[blue!40!white] (6,3) rectangle (7,4);
            \fill[blue!40!white] (7,3) rectangle (8,4);
            \fill[blue!40!white] (8,3) rectangle (9,4);
            
            \fill[blue!40!white] (10,3) rectangle (11,4);
            \fill[blue!40!white] (11,3) rectangle (11.9,4);
            \fill[blue!40!white] (0,4) rectangle (1,4.9);
            \fill[blue!40!white] (1,4) rectangle (2,4.9);
            \fill[blue!40!white] (2,4) rectangle (3,4.9);

            \fill[blue!40!white] (4,4) rectangle (5,4.9);
            \fill[blue!40!white] (5,4) rectangle (6,4.9);
            \fill[blue!40!white] (6,4) rectangle (7,4.9);
            
            \fill[blue!40!white] (8,4) rectangle (9,4.9);
            \fill[blue!40!white] (9,4) rectangle (10,4.9);
            \fill[blue!40!white] (10,4) rectangle (11,4.9);
            \draw[step=1cm,black, thin] (-0.9, -0.9) grid (11.9,4.9);
        \end{tikzpicture}
        \caption{Brick pattern.}\label{fig:brickfigure}
    \end{figure}
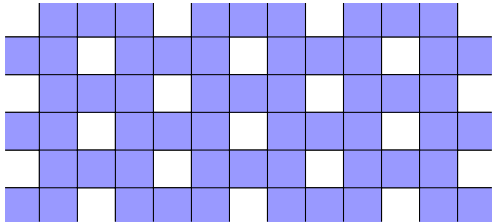    

 Note that the configuration on $\mathbb{Z}^2$ with the brick pattern has the following two properties: 
\begin{itemize}
    \item Each empty lot provides light to exactly three occupied neighbors.
    \item Each occupied lot receives light from exactly one neighbor. 
\end{itemize}
        The construction of finite maximal configurations with the brick pattern is depicted in Figure \ref{brickpatternexamples}. The configurations are obtained by restricting the configuration $C^{\rm brick}$ to a finite grid, and building additional houses on the newly available lots. Among all such restrictions, we choose one that, in the end, yields the highest occupancy.
\begin{figure}
\centering
\begin{subfigure}{.5\textwidth}
  \centering
  \begin{tikzpicture}[scale = 0.5]
            \draw[step=1cm,black,very thin] (0, 0) grid (7,3);
            \fill[blue!40!white] (0,0) rectangle (1,1);
            \fill[blue!40!white] (0,1) rectangle (1,2);
            \fill[blue!40!white] (0,2) rectangle (1,3);
            
            \fill[blue!40!white] (1,0) rectangle (2,1);
            \fill[blue!40!white] (1,2) rectangle (2,3);
            
            \fill[blue!40!white] (2,1) rectangle (3,2);
            \fill[blue!40!white] (2,0) rectangle (3,1);
            \fill[blue!40!white] (2,2) rectangle (3,3);
            
            \fill[blue!40!white] (3,1) rectangle (4,2);
            
            \fill[blue!40!white] (4,0) rectangle (5,1);
            \fill[blue!40!white] (4,1) rectangle (5,2);
			\fill[blue!40!white] (4,2) rectangle (5,3);

            \fill[blue!40!white] (5,0) rectangle (6,1);
            \fill[blue!40!white] (5,2) rectangle (6,3);
            
            \fill[blue!40!white] (6,0) rectangle (7,1);
            \fill[blue!40!white] (6,1) rectangle (7,2);
			\fill[blue!40!white] (6,2) rectangle (7,3);

            \draw[step=1cm,black,very thin] (0, 0) grid (7,3);
        \end{tikzpicture}
        \caption{}\label{brickpattern1}
\end{subfigure}%
\begin{subfigure}{.5\textwidth}
  \centering
  \begin{tikzpicture}[scale = 0.5]
            \draw[step=1cm,black,very thin] (0, 0) grid (6,4);
            \fill[blue!40!white] (0,0) rectangle (1,1);
            \fill[blue!40!white] (0,1) rectangle (1,2);
            \fill[blue!40!white] (0,2) rectangle (1,3);
            \fill[blue!40!white] (0,3) rectangle (1,4);

            \fill[blue!40!white] (1,1) rectangle (2,2);
            \fill[blue!40!white] (1,3) rectangle (2,4);
            
            \fill[blue!40!white] (2,0) rectangle (3,1);
            \fill[blue!40!white] (2,1) rectangle (3,2);
            \fill[blue!40!white] (2,2) rectangle (3,3);
            \fill[blue!40!white] (2,3) rectangle (3,4);
            
            \fill[blue!40!white] (3,0) rectangle (4,1);
            \fill[blue!40!white] (3,2) rectangle (4,3);

            \fill[blue!40!white] (4,0) rectangle (5,1);
            \fill[blue!40!white] (4,1) rectangle (5,2);
            \fill[blue!40!white] (4,2) rectangle (5,3);
            \fill[blue!40!white] (4,3) rectangle (5,4);
            
            \fill[blue!40!white] (5,0) rectangle (6,1);
            \fill[blue!40!white] (5,1) rectangle (6,2);
            \fill[blue!40!white] (5,3) rectangle (6,4);

            \draw[step=1cm,black,very thin] (0, 0) grid (6,4);
        \end{tikzpicture}
  		\caption{}\label{brickpattern2}
\end{subfigure}
\caption{Finite configurations obtained from the brick pattern.}\label{brickpatternexamples}
\end{figure}
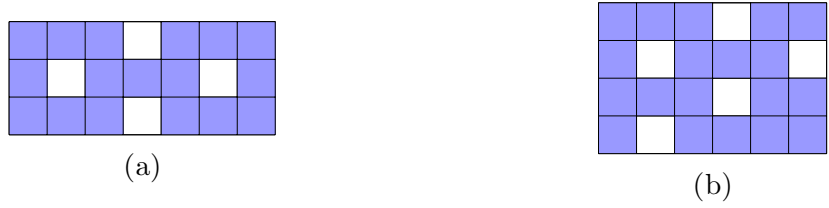  
        In order to determine the occupancy of such configurations, we start by fixing the dimensions $m\geq 2$ and $n \geq 3$. This configuration consists of $\CEIL{\frac{n}{2}}$ fully filled vertical columns of height $m$. The remaining $\FLOOR{\frac{n}{2}}$ columns are filled in a shifted manner with only around a half of the lots occupied. In fact, due to the necessity of shifting the half empty columns, we end up with around a half of them filled with $\FLOOR{\frac{m}{2}}$ houses and the remaining filled with $\CEIL{\frac{m}{2}}$ houses. However, it is always more efficient to arrange the houses in such a way that there are $\CEIL{\frac{\floor{\frac{n}{2}}}{2}}$ columns with $\CEIL{\frac{m}{2}}$ houses and $\FLOOR{\frac{\floor{\frac{n}{2}}}{2}}$ columns with $\FLOOR{\frac{m}{2}}$ houses. We leave it to the interested reader to verify that in the cases $n \equiv 0 \Mod{4}$, independently of $m$; and $n \equiv 2 \Mod{4}, m\equiv 0 \Mod{2}$, there is a possible minor improvement by choosing the restriction such that the bottom right lot can be filled, see Figure \ref{brickpattern2}.

\begin{definition}[Brick pattern function]
        The function that gives us the occupancy $|C|$ of a configuration with the brick pattern is
        \begin{equation*}\small
            \mathcal{O}^{\rm brick}(m,n) := \begin{cases}
            m\CEIL{\frac{n}{2}}+\CEIL{\frac{\floor{\frac{n}{2}}}{2}}\CEIL{\frac{m}{2}} + \FLOOR{\frac{\floor{\frac{n}{2}}}{2}}\FLOOR{\frac{m}{2}} + 1, & {\small\begin{array}{c}
            \text{if } n \equiv 0 \Mod{4} \text{ or }\\
            n \equiv 2 \Mod{4}, m\equiv 0 \Mod{2};
        	\end{array}} \\       
           m\CEIL{\frac{n}{2}}+\CEIL{\frac{\floor{\frac{n}{2}}}{2}}\CEIL{\frac{m}{2}} + \FLOOR{\frac{\floor{\frac{n}{2}}}{2}}\FLOOR{\frac{m}{2}}, &  \mbox{ otherwise. }
            \end{cases}
        \end{equation*}
Additionally, for $n = 2$,  we have $\mathcal{O}^{\rm brick}(m,2) = 2m$. 
    \end{definition}

\subsubsection{Comb pattern}

Another pattern occurring in efficient configurations is the comb pattern, see Figure \ref{fig:combPattern}. The configuration $C^{\rm comb}$ with such a pattern of built houses exhibits a building density of 2/3. 
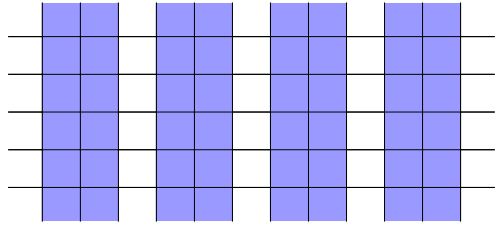
\begin{figure}
        \centering
        \begin{tikzpicture}[scale = 0.5]
            \draw[step=1cm,black,very thin] (-0.9, -0.9) grid (11.9,4.9);
            
            \fill[blue!40!white] (0,-0.9) rectangle (1,0);
            \fill[blue!40!white] (0,0) rectangle (1,1);
            \fill[blue!40!white] (0,1) rectangle (1,2);
            \fill[blue!40!white] (0,2) rectangle (1,3);
            \fill[blue!40!white] (0,3) rectangle (1,4);
            \fill[blue!40!white] (0,4) rectangle (1,4.9);
            
            \fill[blue!40!white] (1,-0.9) rectangle (2,0);
            \fill[blue!40!white] (1,0) rectangle (2,1);
            \fill[blue!40!white] (1,1) rectangle (2,2);
            \fill[blue!40!white] (1,2) rectangle (2,3);
            \fill[blue!40!white] (1,3) rectangle (2,4);
            \fill[blue!40!white] (1,4) rectangle (2,4.9);   
            
            \fill[blue!40!white] (3,-0.9) rectangle (4,0);
            \fill[blue!40!white] (3,0) rectangle (4,1);
            \fill[blue!40!white] (3,1) rectangle (4,2);
            \fill[blue!40!white] (3,2) rectangle (4,3);
            \fill[blue!40!white] (3,3) rectangle (4,4);
            \fill[blue!40!white] (3,4) rectangle (4,4.9);
            
            \fill[blue!40!white] (4,-0.9) rectangle (5,0);
            \fill[blue!40!white] (4,0) rectangle (5,1);
            \fill[blue!40!white] (4,1) rectangle (5,2);
            \fill[blue!40!white] (4,2) rectangle (5,3);
            \fill[blue!40!white] (4,3) rectangle (5,4);
            \fill[blue!40!white] (4,4) rectangle (5,4.9);
            
            \fill[blue!40!white] (6,-0.9) rectangle (7,0);
            \fill[blue!40!white] (6,0) rectangle (7,1);
            \fill[blue!40!white] (6,1) rectangle (7,2);
            \fill[blue!40!white] (6,2) rectangle (7,3);
            \fill[blue!40!white] (6,3) rectangle (7,4);
            \fill[blue!40!white] (6,4) rectangle (7,4.9);
            
            \fill[blue!40!white] (7,-0.9) rectangle (8,0);
            \fill[blue!40!white] (7,0) rectangle (8,1);
            \fill[blue!40!white] (7,1) rectangle (8,2);
            \fill[blue!40!white] (7,2) rectangle (8,3);
            \fill[blue!40!white] (7,3) rectangle (8,4);
            \fill[blue!40!white] (7,4) rectangle (8,4.9);
            
            \fill[blue!40!white] (9,-0.9) rectangle (10,0);
            \fill[blue!40!white] (9,0) rectangle (10,1);
            \fill[blue!40!white] (9,1) rectangle (10,2);
            \fill[blue!40!white] (9,2) rectangle (10,3);
            \fill[blue!40!white] (9,3) rectangle (10,4);
            \fill[blue!40!white] (9,4) rectangle (10,4.9);
            
            \fill[blue!40!white] (10,-0.9) rectangle (11,0);
            \fill[blue!40!white] (10,0) rectangle (11,1);
            \fill[blue!40!white] (10,1) rectangle (11,2);
            \fill[blue!40!white] (10,2) rectangle (11,3);
            \fill[blue!40!white] (10,3) rectangle (11,4);
            \fill[blue!40!white] (10,4) rectangle (11,4.9);
            
            \draw[step=1cm,black,very thin] (-0.9, -0.9) grid (11.9,4.9);
        \end{tikzpicture}
        \caption{Comb pattern.}\label{fig:combPattern}
    \end{figure} 
    Note that the configuration on $\mathbb{Z}^2$ with the comb pattern has the following two properties:
\begin{itemize}
    \item Each empty lot provides light to exactly two of its neighbors.
    \item Each occupied lot receives light from exactly one neighbor. 
\end{itemize}
The way in which we obtain finite maximal configurations with the comb pattern is by restricting $C^{\rm comb}$ to a finite grid and filling all the southernmost lots, see Figure \ref{fig:combpatternexamples}.

    \begin{definition}[Comb pattern function]
    	The function that gives us the occupancy $|C|$ of a configuration with the comb pattern is
        \begin{equation*}
           \mathcal{O}^{\rm comb}(m,n) := \begin{cases}
            n + (m-1)\left(\frac{2}{3}n\right), & \text{if } n \equiv 0 \Mod{3}; \\       
            n + (m-1)\left(\frac{2}{3}(n-1) + 1\right), & \text{if } n \equiv 1 \Mod{3}; \\
            n + (m-1)\left(\frac{2}{3}(n-2) + 2\right), & \text{if } n \equiv 2 \Mod{3};
            \end{cases}
        \end{equation*}
        where $m,n\ge 2$.
    \end{definition}
    
    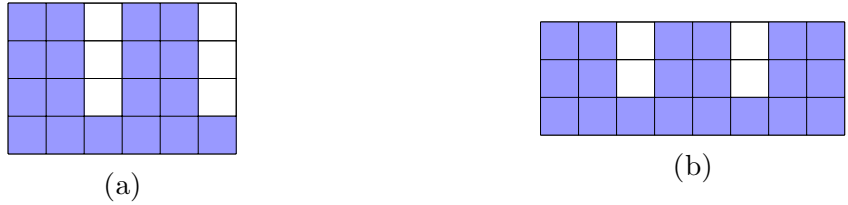
\begin{figure}
\centering
\begin{subfigure}{.5\textwidth}
  \centering
  \begin{tikzpicture}[scale = 0.5]
            \draw[step=1cm,black,very thin] (0, 0) grid (6,4);
            \fill[blue!40!white] (0,0) rectangle (1,1);
            \fill[blue!40!white] (0,1) rectangle (1,2);
            \fill[blue!40!white] (0,2) rectangle (1,3);
            \fill[blue!40!white] (0,3) rectangle (1,4);
            
            \fill[blue!40!white] (1,0) rectangle (2,1);
            \fill[blue!40!white] (1,1) rectangle (2,2);
            \fill[blue!40!white] (1,2) rectangle (2,3);
            \fill[blue!40!white] (1,3) rectangle (2,4);
            
            \fill[blue!40!white] (2,0) rectangle (3,1);
            
            \fill[blue!40!white] (3,0) rectangle (4,1);
            \fill[blue!40!white] (3,1) rectangle (4,2);
            \fill[blue!40!white] (3,2) rectangle (4,3);
            \fill[blue!40!white] (3,3) rectangle (4,4);
            
            \fill[blue!40!white] (4,0) rectangle (5,1);
            \fill[blue!40!white] (4,1) rectangle (5,2);
			\fill[blue!40!white] (4,2) rectangle (5,3);
			\fill[blue!40!white] (4,3) rectangle (5,4);
			
			\fill[blue!40!white] (5,0) rectangle (6,1);

            \draw[step=1cm,black,very thin] (0, 0) grid (6,4);
        \end{tikzpicture}
		\caption{}
\end{subfigure}%
\begin{subfigure}{.5\textwidth}
  \centering
  \begin{tikzpicture}[scale = 0.5]
            \draw[step=1cm,black,very thin] (0, 0) grid (8,3);
            \fill[blue!40!white] (0,0) rectangle (1,1);
            \fill[blue!40!white] (0,1) rectangle (1,2);
            \fill[blue!40!white] (0,2) rectangle (1,3);

            \fill[blue!40!white] (1,0) rectangle (2,1);
            \fill[blue!40!white] (1,1) rectangle (2,2);
            \fill[blue!40!white] (1,2) rectangle (2,3);
       
            \fill[blue!40!white] (2,0) rectangle (3,1);
            
            \fill[blue!40!white] (3,0) rectangle (4,1);
            \fill[blue!40!white] (3,1) rectangle (4,2);
            \fill[blue!40!white] (3,2) rectangle (4,3);
            
            \fill[blue!40!white] (4,0) rectangle (5,1);
            \fill[blue!40!white] (4,1) rectangle (5,2);
			\fill[blue!40!white] (4,2) rectangle (5,3);

            \fill[blue!40!white] (5,0) rectangle (6,1);
            
			\fill[blue!40!white] (6,0) rectangle (7,1);
            \fill[blue!40!white] (6,1) rectangle (7,2);
            \fill[blue!40!white] (6,2) rectangle (7,3);
            
            \fill[blue!40!white] (7,0) rectangle (8,1);
            \fill[blue!40!white] (7,1) rectangle (8,2);
			\fill[blue!40!white] (7,2) rectangle (8,3);
            \draw[step=1cm,black,very thin] (0, 0) grid (8,3);
        \end{tikzpicture}
  	\caption{}
	\end{subfigure}
	\caption{Finite configurations obtained from the comb pattern.}\label{fig:combpatternexamples}
\end{figure}

\begin{remark}
	Based on its intermediate building density of 2/3, it is expected that this pattern will not be occurring in either efficient nor inefficient configurations when both $m$ and $n$ are large. It does, however, occur in some efficient configurations when either $m$ or $n$ are small. This is confirmed in Table \ref{tablicica}.
\end{remark}

\subsubsection{Brick--comb combination}
We note here that the brick pattern and the comb pattern are compatible in a sense that one can design maximal configurations by alternating between these patterns. This is illustrated in Figure \ref{fig:brick_comb_combo}. For certain dimensions $m$ and $n$, the combination of these patterns is more efficient than each of the patterns separately.

\begin{figure}
\centering
\begin{subfigure}{.5\textwidth}
  \centering
  \begin{tikzpicture}[scale = 0.5]
            \draw[step=1cm,black,very thin] (0, 0) grid (10,5);
            \fill[blue!40!white] (0,0) rectangle (1,1);
            \fill[blue!40!white] (0,1) rectangle (1,2);
            \fill[blue!40!white] (0,2) rectangle (1,3);
            \fill[blue!40!white] (0,3) rectangle (1,4);
            \fill[blue!40!white] (0,4) rectangle (1,5);
            
            \fill[blue!40!white] (1,0) rectangle (2,1);
            \fill[blue!40!white] (1,2) rectangle (2,3);
            \fill[blue!40!white] (1,4) rectangle (2,5);
            
            \fill[blue!40!white] (2,1) rectangle (3,2);
            \fill[blue!40!white] (2,0) rectangle (3,1);
            \fill[blue!40!white] (2,2) rectangle (3,3);
            \fill[blue!40!white] (2,3) rectangle (3,4);
            \fill[blue!40!white] (2,4) rectangle (3,5);
            
            \fill[blue!40!white] (3,1) rectangle (4,2);
            \fill[blue!40!white] (3,3) rectangle (4,4);
            
            \fill[blue!40!white] (4,0) rectangle (5,1);
            \fill[blue!40!white] (4,1) rectangle (5,2);
			\fill[blue!40!white] (4,2) rectangle (5,3);
			\fill[blue!40!white] (4,3) rectangle (5,4);
			\fill[blue!40!white] (4,4) rectangle (5,5);

            \fill[blue!40!white] (5,0) rectangle (6,1);
            \fill[blue!40!white] (5,2) rectangle (6,3);
            \fill[blue!40!white] (5,4) rectangle (6,5);
            
            \fill[blue!40!white] (6,0) rectangle (7,1);
            \fill[blue!40!white] (6,1) rectangle (7,2);
			\fill[blue!40!white] (6,2) rectangle (7,3);
			\fill[blue!40!white] (6,3) rectangle (7,4);
			\fill[blue!40!white] (6,4) rectangle (7,5);
			
			\fill[blue!40!white] (7,1) rectangle (8,2);
            \fill[blue!40!white] (7,3) rectangle (8,4);
            
            \fill[blue!40!white] (8,0) rectangle (9,1);
            \fill[blue!40!white] (8,1) rectangle (9,2);
			\fill[blue!40!white] (8,2) rectangle (9,3);
			\fill[blue!40!white] (8,3) rectangle (9,4);
			\fill[blue!40!white] (8,4) rectangle (9,5);

            \fill[blue!40!white] (9,0) rectangle (10,1);
            \fill[blue!40!white] (9,2) rectangle (10,3);
            \fill[blue!40!white] (9,4) rectangle (10,5);

            \draw[step=1cm,black,very thin] (0, 0) grid (10,5);
        \end{tikzpicture}
  \caption{Brick pattern only.}\label{fig:brickpatternonly}
\end{subfigure}%
\begin{subfigure}{.5\textwidth}
  \centering
  \begin{tikzpicture}[scale = 0.5]
            \draw[step=1cm,black,very thin] (0, 0) grid (10,5);
            \fill[blue!40!white] (0,0) rectangle (1,1);
            \fill[blue!40!white] (0,1) rectangle (1,2);
            \fill[blue!40!white] (0,2) rectangle (1,3);
            \fill[blue!40!white] (0,3) rectangle (1,4);
            \fill[blue!40!white] (0,4) rectangle (1,5);
            
            \fill[blue!40!white] (1,0) rectangle (2,1);
            \fill[blue!40!white] (1,1) rectangle (2,2);
            \fill[blue!40!white] (1,2) rectangle (2,3);
            \fill[blue!40!white] (1,3) rectangle (2,4);
            \fill[blue!40!white] (1,4) rectangle (2,5);
            
            \fill[blue!40!white] (2,0) rectangle (3,1);
           
            \fill[blue!40!white] (3,0) rectangle (4,1);
            \fill[blue!40!white] (3,1) rectangle (4,2);
            \fill[blue!40!white] (3,2) rectangle (4,3);
            \fill[blue!40!white] (3,3) rectangle (4,4);
            \fill[blue!40!white] (3,4) rectangle (4,5);
            
            \fill[blue!40!white] (4,0) rectangle (5,1);
            \fill[blue!40!white] (4,1) rectangle (5,2);
			\fill[blue!40!white] (4,2) rectangle (5,3);
			\fill[blue!40!white] (4,3) rectangle (5,4);
			\fill[blue!40!white] (4,4) rectangle (5,5);

            \fill[blue!40!white] (5,0) rectangle (6,1);
            
            \fill[blue!40!white] (6,0) rectangle (7,1);
            \fill[blue!40!white] (6,1) rectangle (7,2);
			\fill[blue!40!white] (6,2) rectangle (7,3);
			\fill[blue!40!white] (6,3) rectangle (7,4);
			\fill[blue!40!white] (6,4) rectangle (7,5);
			
			\fill[blue!40!white] (7,0) rectangle (8,1);
			\fill[blue!40!white] (7,1) rectangle (8,2);
			\fill[blue!40!white] (7,2) rectangle (8,3);
            \fill[blue!40!white] (7,3) rectangle (8,4);
            \fill[blue!40!white] (7,4) rectangle (8,5);
            
            \fill[blue!40!white] (8,0) rectangle (9,1);

            \fill[blue!40!white] (9,0) rectangle (10,1);
            \fill[blue!40!white] (9,1) rectangle (10,2);
            \fill[blue!40!white] (9,2) rectangle (10,3);
            \fill[blue!40!white] (9,3) rectangle (10,4);
            \fill[blue!40!white] (9,4) rectangle (10,5);

            \draw[step=1cm,black,very thin] (0, 0) grid (10,5);
        \end{tikzpicture}
  		\caption{Comb pattern only.}\label{fig:combonly}
\end{subfigure}
\begin{subfigure}{.5\textwidth}
  \centering
   \begin{tikzpicture}[scale = 0.5]
            \draw[step=1cm,black,very thin] (0, 0) grid (10,5);
            \fill[blue!40!white] (0,0) rectangle (1,1);
            \fill[blue!40!white] (0,1) rectangle (1,2);
            \fill[blue!40!white] (0,2) rectangle (1,3);
            \fill[blue!40!white] (0,3) rectangle (1,4);
            \fill[blue!40!white] (0,4) rectangle (1,5);
            
            \fill[blue!40!white] (1,0) rectangle (2,1);
            \fill[blue!40!white] (1,2) rectangle (2,3);
            \fill[blue!40!white] (1,4) rectangle (2,5);
            
            \fill[blue!40!white] (2,1) rectangle (3,2);
            \fill[blue!40!white] (2,0) rectangle (3,1);
            \fill[blue!40!white] (2,2) rectangle (3,3);
            \fill[blue!40!white] (2,3) rectangle (3,4);
            \fill[blue!40!white] (2,4) rectangle (3,5);
            
            \fill[blue!40!white] (3,1) rectangle (4,2);
            \fill[blue!40!white] (3,3) rectangle (4,4);
            
            \fill[blue!40!white] (4,0) rectangle (5,1);
            \fill[blue!40!white] (4,1) rectangle (5,2);
			\fill[blue!40!white] (4,2) rectangle (5,3);
			\fill[blue!40!white] (4,3) rectangle (5,4);
			\fill[blue!40!white] (4,4) rectangle (5,5);

            \fill[blue!40!white] (5,0) rectangle (6,1);
            \fill[blue!40!white] (5,2) rectangle (6,3);
            \fill[blue!40!white] (5,4) rectangle (6,5);
            
            \fill[blue!40!white] (6,0) rectangle (7,1);
            \fill[blue!40!white] (6,1) rectangle (7,2);
			\fill[blue!40!white] (6,2) rectangle (7,3);
			\fill[blue!40!white] (6,3) rectangle (7,4);
			\fill[blue!40!white] (6,4) rectangle (7,5);
			
			\fill[blue!40!white] (7,0) rectangle (8,1);
            
            \fill[blue!40!white] (8,0) rectangle (9,1);
            \fill[blue!40!white] (8,1) rectangle (9,2);
			\fill[blue!40!white] (8,2) rectangle (9,3);
			\fill[blue!40!white] (8,3) rectangle (9,4);
			\fill[blue!40!white] (8,4) rectangle (9,5);

            \fill[blue!40!white] (9,0) rectangle (10,1);
            \fill[blue!40!white] (9,1) rectangle (10,2);
            \fill[blue!40!white] (9,2) rectangle (10,3);
            \fill[blue!40!white] (9,3) rectangle (10,4);
            \fill[blue!40!white] (9,4) rectangle (10,5);

            \draw[step=1cm,black,very thin] (0, 0) grid (10,5);
        \end{tikzpicture}
        \caption{The combination of the two patterns.}\label{fig:combination}
		\end{subfigure}
		
	\caption{Consider the tract of land with dimensions $5\times 10$. While the configuration with the brick pattern exhibits the occupancy of $38$, as well as the configuration with the comb pattern, the combination, on the other hand, exhibits the occupancy of $39$.}\label{fig:brick_comb_combo}
\end{figure}
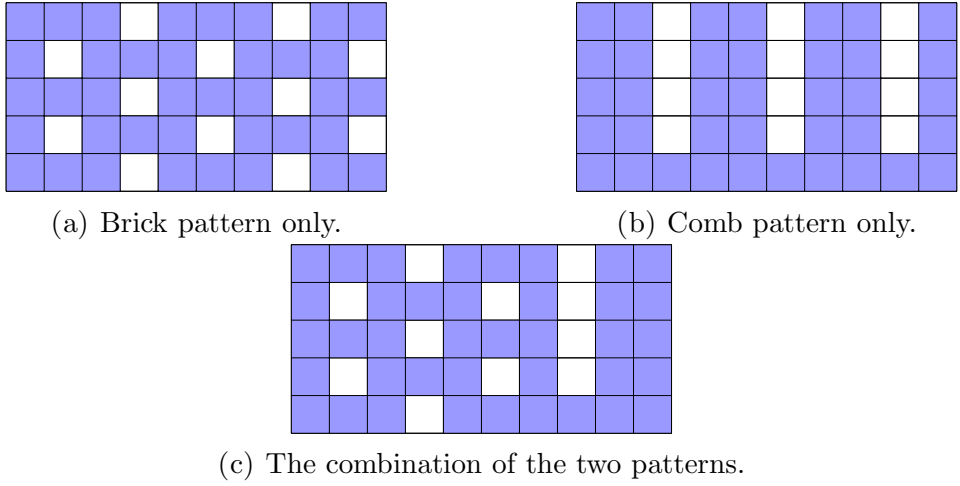

\subsection{Patterns occurring in inefficient configurations}

\subsubsection{Rake pattern}
The rake pattern is a periodic pattern similar to the comb pattern but with bigger gaps. However, the configuration $C^{\rm rake}$ with the rake pattern, see Figure \ref{fig:rake}, yields the minimal possible building density of $1/2$.
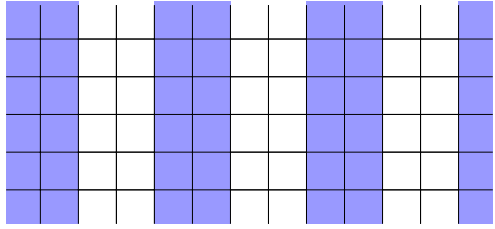
\begin{figure}
        \centering
        \begin{tikzpicture}[scale = 0.5]
            \draw[step=1cm,black,very thin] (-0.9, -0.9) grid (11.9,4.9);
            
            \fill[blue!40!white] (-0.9,-0.9) rectangle (0,0);
            \fill[blue!40!white] (-0.9,0) rectangle (0,1);
            \fill[blue!40!white] (-0.9,1) rectangle (0,2);
            \fill[blue!40!white] (-0.9,2) rectangle (0,3);
            \fill[blue!40!white] (-0.9,3) rectangle (0,4);
            \fill[blue!40!white] (-0.9,4) rectangle (0,5);
            
            \fill[blue!40!white] (0,-0.9) rectangle (1,0);
            \fill[blue!40!white] (0,0) rectangle (1,1);
            \fill[blue!40!white] (0,1) rectangle (1,2);
            \fill[blue!40!white] (0,2) rectangle (1,3);
            \fill[blue!40!white] (0,3) rectangle (1,4);
            \fill[blue!40!white] (0,4) rectangle (1,5);

            \fill[blue!40!white] (3,-0.9) rectangle (4,0);
            \fill[blue!40!white] (3,0) rectangle (4,1);
            \fill[blue!40!white] (3,1) rectangle (4,2);
            \fill[blue!40!white] (3,2) rectangle (4,3);
            \fill[blue!40!white] (3,3) rectangle (4,4);
            \fill[blue!40!white] (3,4) rectangle (4,5);
            
            \fill[blue!40!white] (4,-0.9) rectangle (5,0);
            \fill[blue!40!white] (4,0) rectangle (5,1);
            \fill[blue!40!white] (4,1) rectangle (5,2);
            \fill[blue!40!white] (4,2) rectangle (5,3);
            \fill[blue!40!white] (4,3) rectangle (5,4);
            \fill[blue!40!white] (4,4) rectangle (5,5);

            \fill[blue!40!white] (7,-0.9) rectangle (8,0);
            \fill[blue!40!white] (7,0) rectangle (8,1);
            \fill[blue!40!white] (7,1) rectangle (8,2);
            \fill[blue!40!white] (7,2) rectangle (8,3);
            \fill[blue!40!white] (7,3) rectangle (8,4);
            \fill[blue!40!white] (7,4) rectangle (8,5);
            
            \fill[blue!40!white] (8,-0.9) rectangle (9,0);
            \fill[blue!40!white] (8,0) rectangle (9,1);
            \fill[blue!40!white] (8,1) rectangle (9,2);
            \fill[blue!40!white] (8,2) rectangle (9,3);
            \fill[blue!40!white] (8,3) rectangle (9,4);
            \fill[blue!40!white] (8,4) rectangle (9,5);
            
            \fill[blue!40!white] (11,-0.9) rectangle (11.9,0);
            \fill[blue!40!white] (11,0) rectangle (11.9,1);
            \fill[blue!40!white] (11,1) rectangle (11.9,2);
            \fill[blue!40!white] (11,2) rectangle (11.9,3);
            \fill[blue!40!white] (11,3) rectangle (11.9,4);
            \fill[blue!40!white] (11,4) rectangle (11.9,5);
            
            \draw[step=1cm,black,very thin] (-0.9, -0.9) grid (11.9,4.9);
        \end{tikzpicture}
        \caption{Rake pattern}\label{fig:rake}
    \end{figure}   
    Note that the configuration on $\mathbb{Z}^2$ with the rake pattern has the following two properties:
\begin{itemize}
    \item Each empty lot provides light to only one of its neighbors.
    \item Each occupied lot receives light from only one of its neighbors. 
\end{itemize}  

The finite maximal configurations with the rake pattern are obtained by restricting $C^{\rm rake}$ to a finite grid, and building additional houses on the newly available lots. Among all such restrictions, we choose one that, in the end, yields the lowest occupancy, see Figure \ref{fig:rakepatternexamples}. 

    \begin{definition}[Rake pattern function]
        The function that gives us the occupancy $|C|$ of a configuration with the rake pattern is
        \begin{equation*}
            \mathcal{O}^{\rm rake}(m,n) := \begin{cases}
            n + (m-1)\left(\frac{1}{2}n\right), & \text{if } n \equiv 0 \Mod{4}; \\       
            n + (m-1)\left(\frac{1}{2}(n-1) + 1\right), & \text{if } n \equiv 1 \Mod{4}; \\
            n + (m-1)\left(\frac{1}{2}(n-2) + 2\right), & \text{if } n \equiv 2 \Mod{4}; \\
            n + (m-1)\left(\frac{1}{2}(n-3) + 2\right), & \text{if } n \equiv 3 \Mod{4};
            \end{cases}
        \end{equation*}
    \end{definition}
	where $m,n\ge 2$.
    
        \begin{figure}
\centering
\begin{subfigure}{.5\textwidth}
  \centering
  \begin{tikzpicture}[scale = 0.5]
            \draw[step=1cm,black,very thin] (0, 1) grid (8,5);
            \fill[blue!40!white] (0,1) rectangle (1,2);

            \fill[blue!40!white] (1,1) rectangle (2,2);
            \fill[blue!40!white] (1,2) rectangle (2,3);
            \fill[blue!40!white] (1,3) rectangle (2,4);
            \fill[blue!40!white] (1,4) rectangle (2,5);
            
            \fill[blue!40!white] (2,1) rectangle (3,2);
            \fill[blue!40!white] (2,2) rectangle (3,3);
            \fill[blue!40!white] (2,3) rectangle (3,4);
            \fill[blue!40!white] (2,4) rectangle (3,5);
            
            \fill[blue!40!white] (3,1) rectangle (4,2);
            
            \fill[blue!40!white] (4,1) rectangle (5,2);
            
            \fill[blue!40!white] (5,1) rectangle (6,2);
            \fill[blue!40!white] (5,2) rectangle (6,3);
            \fill[blue!40!white] (5,3) rectangle (6,4);
            \fill[blue!40!white] (5,4) rectangle (6,5);
            
            \fill[blue!40!white] (6,1) rectangle (7,2);
            \fill[blue!40!white] (6,2) rectangle (7,3);
			\fill[blue!40!white] (6,3) rectangle (7,4);
			\fill[blue!40!white] (6,4) rectangle (7,5);
			
            \fill[blue!40!white] (7,1) rectangle (8,2);

            \draw[step=1cm,black,very thin] (0, 1) grid (8,5);
        \end{tikzpicture}
  \caption{}
\end{subfigure}%
\begin{subfigure}{.5\textwidth}
  \centering
  \begin{tikzpicture}[scale = 0.5]
            \draw[step=1cm,black,very thin] (-1, 1) grid (12,4);
            \fill[blue!40!white] (-1,1) rectangle (0,2);
            \fill[blue!40!white] (-1,2) rectangle (0,3);
            \fill[blue!40!white] (-1,3) rectangle (0,4);
            
            \fill[blue!40!white] (0,1) rectangle (1,2);

            \fill[blue!40!white] (1,1) rectangle (2,2);
            \fill[blue!40!white] (1,2) rectangle (2,3);
            \fill[blue!40!white] (1,3) rectangle (2,4);
            
            \fill[blue!40!white] (2,1) rectangle (3,2);
            \fill[blue!40!white] (2,2) rectangle (3,3);
            \fill[blue!40!white] (2,3) rectangle (3,4);
            
            \fill[blue!40!white] (3,1) rectangle (4,2);
            
            \fill[blue!40!white] (4,1) rectangle (5,2);
            
            \fill[blue!40!white] (5,1) rectangle (6,2);
            \fill[blue!40!white] (5,2) rectangle (6,3);
            \fill[blue!40!white] (5,3) rectangle (6,4);
            
            \fill[blue!40!white] (6,1) rectangle (7,2);
            \fill[blue!40!white] (6,2) rectangle (7,3);
			\fill[blue!40!white] (6,3) rectangle (7,4);
			
			\fill[blue!40!white] (7,1) rectangle (8,2);
            
            \fill[blue!40!white] (8,1) rectangle (9,2);
            
            \fill[blue!40!white] (9,1) rectangle (10,2);
            \fill[blue!40!white] (9,2) rectangle (10,3);
            \fill[blue!40!white] (9,3) rectangle (10,4);
            
            \fill[blue!40!white] (10,1) rectangle (11,2);
            \fill[blue!40!white] (10,2) rectangle (11,3);
			\fill[blue!40!white] (10,3) rectangle (11,4);
			
            \fill[blue!40!white] (10,1) rectangle (11,2);
            
			\fill[blue!40!white] (11,1) rectangle (12,2);
            \draw[step=1cm,black,very thin] (-1, 1) grid (12,4);
        \end{tikzpicture}
  \caption{}
\end{subfigure}
\caption{Finite configurations obtained from the rake pattern.}\label{fig:rakepatternexamples}
\end{figure}
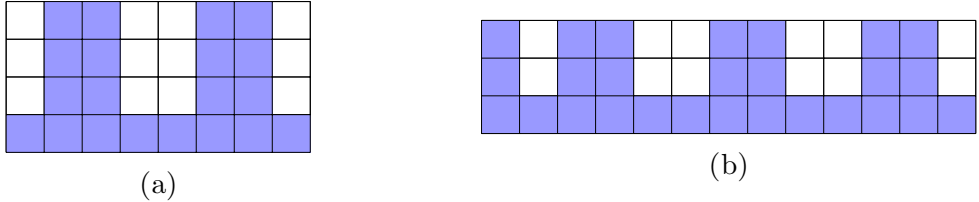

\subsubsection{Stripe pattern}
The stripe pattern is a periodic pattern consisting of horizontal lines of built houses depicted in Figure \ref{fig:stripefigure}. The configuration $C^{\rm stripe}$ on $\mathbb{Z}^2$ with the stripe pattern also yields the minimal possible building density of $1/2$, the same as the rake pattern.
\begin{figure}
        \centering
        \begin{tikzpicture}[scale = 0.5]
            \draw[step=1cm,black,very thin] (-0.9, -0.9) grid (11.9,5.9);
            
            \fill[blue!40!white] (-0.9,-0.9) rectangle (0,0);
            \fill[blue!40!white] (-0.9,1) rectangle (0,2);
            \fill[blue!40!white] (-0.9,3) rectangle (0,4);
            \fill[blue!40!white] (-0.9,5) rectangle (0,5.9);

            \fill[blue!40!white] (0,-0.9) rectangle (1,0);
            \fill[blue!40!white] (0,1) rectangle (1,2);
            \fill[blue!40!white] (0,3) rectangle (1,4);
            \fill[blue!40!white] (0,5) rectangle (1,5.9);

            \fill[blue!40!white] (1,-0.9) rectangle (2,0);
            \fill[blue!40!white] (1,1) rectangle (2,2);
            \fill[blue!40!white] (1,3) rectangle (2,4);
            \fill[blue!40!white] (1,5) rectangle (2,5.9);

            \fill[blue!40!white] (2,-0.9) rectangle (3,0);
            \fill[blue!40!white] (2,1) rectangle (3,2);
            \fill[blue!40!white] (2,3) rectangle (3,4);
            \fill[blue!40!white] (2,5) rectangle (3,5.9);

            \fill[blue!40!white] (3,-0.9) rectangle (4,0);
            \fill[blue!40!white] (3,1) rectangle (4,2);
            \fill[blue!40!white] (3,3) rectangle (4,4);
            \fill[blue!40!white] (3,5) rectangle (4,5.9);

            \fill[blue!40!white] (4,-0.9) rectangle (5,0);
            \fill[blue!40!white] (4,1) rectangle (5,2);
            \fill[blue!40!white] (4,3) rectangle (5,4);
            \fill[blue!40!white] (4,5) rectangle (5,5.9);

            \fill[blue!40!white] (5,-0.9) rectangle (6,0);
            \fill[blue!40!white] (5,1) rectangle (6,2);
            \fill[blue!40!white] (5,3) rectangle (6,4);
            \fill[blue!40!white] (5,5) rectangle (6,5.9);

            \fill[blue!40!white] (6,-0.9) rectangle (7,0);
            \fill[blue!40!white] (6,1) rectangle (7,2);
            \fill[blue!40!white] (6,3) rectangle (7,4);
            \fill[blue!40!white] (6,5) rectangle (7,5.9);

            \fill[blue!40!white] (7,-0.9) rectangle (8,0);
            \fill[blue!40!white] (7,1) rectangle (8,2);
            \fill[blue!40!white] (7,3) rectangle (8,4);
            \fill[blue!40!white] (7,5) rectangle (8,5.9);

            \fill[blue!40!white] (8,-0.9) rectangle (9,0);
            \fill[blue!40!white] (8,1) rectangle (9,2);
            \fill[blue!40!white] (8,3) rectangle (9,4);
            \fill[blue!40!white] (7,5) rectangle (8,5.9);

            \fill[blue!40!white] (8,-0.9) rectangle (9,0);
            \fill[blue!40!white] (8,1) rectangle (9,2);
            \fill[blue!40!white] (8,3) rectangle (9,4);
            \fill[blue!40!white] (8,5) rectangle (9,5.9);

            \fill[blue!40!white] (9,-0.9) rectangle (10,0);
            \fill[blue!40!white] (9,1) rectangle (10,2);
            \fill[blue!40!white] (9,3) rectangle (10,4);
            \fill[blue!40!white] (9,5) rectangle (10,5.9);

            \fill[blue!40!white] (10,-0.9) rectangle (11,0);
            \fill[blue!40!white] (10,1) rectangle (11,2);
            \fill[blue!40!white] (10,3) rectangle (11,4);
            \fill[blue!40!white] (10,5) rectangle (11,5.9);

            \fill[blue!40!white] (11,-0.9) rectangle (11.9,0);
            \fill[blue!40!white] (11,1) rectangle (11.9,2);
            \fill[blue!40!white] (11,3) rectangle (11.9,4);
            \fill[blue!40!white] (11,5) rectangle (11.9,5.9);
            
            \draw[step=1cm,black,very thin] (-0.9, -0.9) grid (11.9,5.9);
        \end{tikzpicture}
        \caption{Stripe pattern.}\label{fig:stripefigure}
    \end{figure}
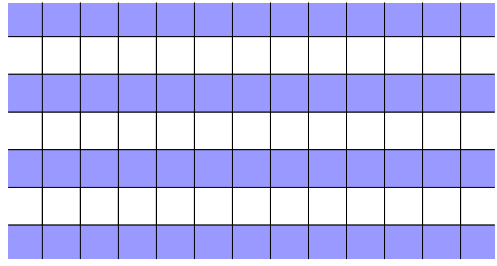   
    Note that the configuration on $\mathbb{Z}^2$ with the stripe pattern has the following two properties:
	\begin{itemize}
	    \item Each empty lot provides light to only one of its neighbors.
	    \item Each occupied lot receives light from only one of its neighbors. 
	\end{itemize}  
\begin{remark}
 Notice that both of the above mentioned properties are shared by both the stripe and rake patterns. However, not all periodic patterns with the building density of $1/2$ obey these two conditions, see the check pattern in Section \ref{subsubsec:check}.
\end{remark}

The finite maximal configurations with the stripe pattern are obtained by restricting $C^{\rm stripe}$ to a finite grid and building additional houses on the newly available lots. Among all such restrictions, we choose one that, in the end, yields the lowest occupancy, see Figure \ref{fig:stripepatternexamples}.

    \begin{definition}[Stripe pattern function]
        The function that gives us the occupancy $|C|$ of a configuration with the stripe pattern is
        \begin{equation*}\small
            \mathcal{O}^{\rm stripe}(m,n) := \begin{cases}
            2m + \FLOOR{\frac{m}{2}}(n-2), & \text{if } m \equiv 0 \Mod2; \\
            2m + \FLOOR{\frac{m}{2}}(n-2)+\left(\frac{1}{2}n\right), & \text{if } m \equiv 1 \Mod2, n \equiv 0 \Mod4; \\
            2m + \FLOOR{\frac{m}{2}}(n-2)+\left(\frac{1}{2}(n-1) + 1\right), & \text{if } m \equiv 1 \Mod2, n \equiv 1 \Mod4; \\
            2m + \FLOOR{\frac{m}{2}}(n-2)+\left(\frac{1}{2}(n-2) + 2\right), & \text{if } m \equiv 1 \Mod2, n \equiv 2 \Mod4; \\
            2m + \FLOOR{\frac{m}{2}}(n-2)+\left(\frac{1}{2}(n-3) + 2\right), & \text{if } m \equiv 1 \Mod2, n \equiv 3 \Mod4;
            \end{cases}
        \end{equation*}
        where $m,n\ge 2$.
    \end{definition} 
    
        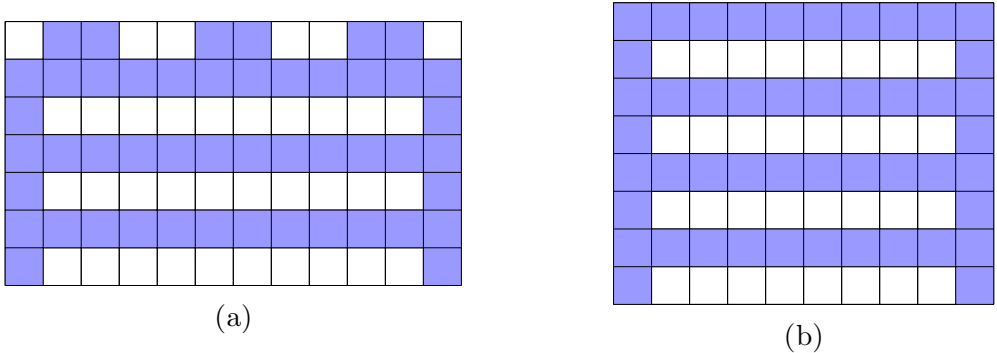
\begin{figure}
\centering
\begin{subfigure}{.5\textwidth}
  \centering
         \begin{tikzpicture}[scale = 0.5]
            \draw[step=1cm,black,very thin] (0, 0) grid (12,7);
            
            \fill[blue!40!white] (0,0) rectangle (1,1);
            \fill[blue!40!white] (0,1) rectangle (1,2);
            \fill[blue!40!white] (0,2) rectangle (1,3);
            \fill[blue!40!white] (0,3) rectangle (1,4);
            \fill[blue!40!white] (0,4) rectangle (1,5);
            \fill[blue!40!white] (0,5) rectangle (1,6);

            \fill[blue!40!white] (1,1) rectangle (2,2);
            \fill[blue!40!white] (1,3) rectangle (2,4);
            \fill[blue!40!white] (1,5) rectangle (2,6);
            \fill[blue!40!white] (1,6) rectangle (2,7);

            \fill[blue!40!white] (2,1) rectangle (3,2);
            \fill[blue!40!white] (2,3) rectangle (3,4);
            \fill[blue!40!white] (2,5) rectangle (3,6);
            \fill[blue!40!white] (2,6) rectangle (3,7);

            \fill[blue!40!white] (3,1) rectangle (4,2);
            \fill[blue!40!white] (3,3) rectangle (4,4);
            \fill[blue!40!white] (3,5) rectangle (4,6);

            \fill[blue!40!white] (4,1) rectangle (5,2);
            \fill[blue!40!white] (4,3) rectangle (5,4);
            \fill[blue!40!white] (4,5) rectangle (5,6);

            \fill[blue!40!white] (5,1) rectangle (6,2);
            \fill[blue!40!white] (5,3) rectangle (6,4);
            \fill[blue!40!white] (5,5) rectangle (6,6);
            \fill[blue!40!white] (5,6) rectangle (6,7);

            \fill[blue!40!white] (6,1) rectangle (7,2);
            \fill[blue!40!white] (6,3) rectangle (7,4);
            \fill[blue!40!white] (6,5) rectangle (7,6);
            \fill[blue!40!white] (6,6) rectangle (7,7);

            \fill[blue!40!white] (7,1) rectangle (8,2);
            \fill[blue!40!white] (7,3) rectangle (8,4);
            \fill[blue!40!white] (7,5) rectangle (8,6);

            \fill[blue!40!white] (8,1) rectangle (9,2);
            \fill[blue!40!white] (8,3) rectangle (9,4);
            \fill[blue!40!white] (7,5) rectangle (8,6);

            \fill[blue!40!white] (8,1) rectangle (9,2);
            \fill[blue!40!white] (8,3) rectangle (9,4);
            \fill[blue!40!white] (8,5) rectangle (9,6);

            \fill[blue!40!white] (9,1) rectangle (10,2);
            \fill[blue!40!white] (9,3) rectangle (10,4);
            \fill[blue!40!white] (9,5) rectangle (10,6);
            \fill[blue!40!white] (9,6) rectangle (10,7);

            \fill[blue!40!white] (10,1) rectangle (11,2);
            \fill[blue!40!white] (10,3) rectangle (11,4);
            \fill[blue!40!white] (10,5) rectangle (11,6);
            \fill[blue!40!white] (10,6) rectangle (11,7);

            \fill[blue!40!white] (11,0) rectangle (12,1);
            \fill[blue!40!white] (11,1) rectangle (12,2);
            \fill[blue!40!white] (11,2) rectangle (12,3);
            \fill[blue!40!white] (11,3) rectangle (12,4);
            \fill[blue!40!white] (11,4) rectangle (12,5);
            \fill[blue!40!white] (11,5) rectangle (12,6);
            
            \draw[step=1cm,black,very thin] (0, 0) grid (12,7);
        \end{tikzpicture}
		\caption{}\label{fig:Stripe_a}
\end{subfigure}%
\begin{subfigure}{.5\textwidth}
  \centering
          \begin{tikzpicture}[scale = 0.5]
            \draw[step=1cm,black,very thin] (0, -2) grid (10,6);

            \fill[blue!40!white] (0,-2) rectangle (1,-1);
            \fill[blue!40!white] (0,-1) rectangle (1,0);
            \fill[blue!40!white] (0,0) rectangle (1,1);
            \fill[blue!40!white] (0,1) rectangle (1,2);
            \fill[blue!40!white] (0,2) rectangle (1,3);
            \fill[blue!40!white] (0,3) rectangle (1,4);
            \fill[blue!40!white] (0,4) rectangle (1,5);
            \fill[blue!40!white] (0,5) rectangle (1,6);

            \fill[blue!40!white] (1,-1) rectangle (2,0);
            \fill[blue!40!white] (1,1) rectangle (2,2);
            \fill[blue!40!white] (1,3) rectangle (2,4);
            \fill[blue!40!white] (1,5) rectangle (2,6);

            \fill[blue!40!white] (2,-1) rectangle (3,0);
            \fill[blue!40!white] (2,1) rectangle (3,2);
            \fill[blue!40!white] (2,3) rectangle (3,4);
            \fill[blue!40!white] (2,5) rectangle (3,6);

            \fill[blue!40!white] (3,-1) rectangle (4,0);
            \fill[blue!40!white] (3,1) rectangle (4,2);
            \fill[blue!40!white] (3,3) rectangle (4,4);
            \fill[blue!40!white] (3,5) rectangle (4,6);

            \fill[blue!40!white] (4,-1) rectangle (5,0);
            \fill[blue!40!white] (4,1) rectangle (5,2);
            \fill[blue!40!white] (4,3) rectangle (5,4);
            \fill[blue!40!white] (4,5) rectangle (5,6);

            \fill[blue!40!white] (5,-1) rectangle (6,0);
            \fill[blue!40!white] (5,1) rectangle (6,2);
            \fill[blue!40!white] (5,3) rectangle (6,4);
            \fill[blue!40!white] (5,5) rectangle (6,6);

            \fill[blue!40!white] (6,-1) rectangle (7,0);
            \fill[blue!40!white] (6,1) rectangle (7,2);
            \fill[blue!40!white] (6,3) rectangle (7,4);
            \fill[blue!40!white] (6,5) rectangle (7,6);

            \fill[blue!40!white] (7,-1) rectangle (8,0);
            \fill[blue!40!white] (7,1) rectangle (8,2);
            \fill[blue!40!white] (7,3) rectangle (8,4);
            \fill[blue!40!white] (7,5) rectangle (8,6);

            \fill[blue!40!white] (8,-1) rectangle (9,0);
            \fill[blue!40!white] (8,1) rectangle (9,2);
            \fill[blue!40!white] (8,3) rectangle (9,4);
            \fill[blue!40!white] (7,5) rectangle (8,6);

            \fill[blue!40!white] (8,-1) rectangle (9,0);
            \fill[blue!40!white] (8,1) rectangle (9,2);
            \fill[blue!40!white] (8,3) rectangle (9,4);
            \fill[blue!40!white] (8,5) rectangle (9,6);

            \fill[blue!40!white] (9,-2) rectangle (10,-1);
            \fill[blue!40!white] (9,-1) rectangle (10,0);
            \fill[blue!40!white] (9,0) rectangle (10,1);
            \fill[blue!40!white] (9,1) rectangle (10,2);
            \fill[blue!40!white] (9,2) rectangle (10,3);
            \fill[blue!40!white] (9,3) rectangle (10,4);
            \fill[blue!40!white] (9,4) rectangle (10,5);
            \fill[blue!40!white] (9,5) rectangle (10,6);

            \draw[step=1cm,black,very thin] (0, -2) grid (10,6);
        \end{tikzpicture}
		  \caption{}
\end{subfigure}
\caption{Finite configurations obtained from the stripe pattern.}\label{fig:stripepatternexamples}
\end{figure}  

\begin{remark}\label{rem:stripe_firstrow}
	Note that, in the case when $m$ is odd, in order to yield the lowest occupancy, we need to fill the first row with a part of the rake pattern, see Figure \ref{fig:Stripe_a}.
\end{remark}

\subsubsection{Rake--stripe combination}\label{subsect:rakeStripe}
Inspired by Remark \ref{rem:stripe_firstrow} and forthcoming Lemma \ref{lm:n+2}, we form a combination of the rake and stripe patterns, see Figure \ref{fig:comb_rake_stripe}. This combination yields lower occupancy than each of the patterns separately. Moreover, in Theorem \ref{tm:Imn}, we show that this combination attains the lowest occupancy possible among all maximal configurations.

\begin{figure}
\centering
\begin{subfigure}{.5\textwidth}
  \centering
  \begin{tikzpicture}[scale = 0.5]
            \draw[step=1cm,black,very thin] (3, 0) grid (11,6);

            \fill[blue!40!white] (3,0) rectangle (4,1);
            
            \fill[blue!40!white] (4,0) rectangle (5,1);
            \fill[blue!40!white] (4,1) rectangle (5,2);
			\fill[blue!40!white] (4,2) rectangle (5,3);
			\fill[blue!40!white] (4,3) rectangle (5,4);
			\fill[blue!40!white] (4,4) rectangle (5,5);
			\fill[blue!40!white] (4,5) rectangle (5,6);

            \fill[blue!40!white] (5,0) rectangle (6,1);
            \fill[blue!40!white] (5,1) rectangle (6,2);
			\fill[blue!40!white] (5,2) rectangle (6,3);
			\fill[blue!40!white] (5,3) rectangle (6,4);
			\fill[blue!40!white] (5,4) rectangle (6,5);
			\fill[blue!40!white] (5,5) rectangle (6,6);
			
			\fill[blue!40!white] (6,0) rectangle (7,1);
			\fill[blue!40!white] (7,0) rectangle (8,1);
			
			\fill[blue!40!white] (8,0) rectangle (9,1);
			\fill[blue!40!white] (8,1) rectangle (9,2);
			\fill[blue!40!white] (8,2) rectangle (9,3);
            \fill[blue!40!white] (8,3) rectangle (9,4);
            \fill[blue!40!white] (8,4) rectangle (9,5);
            \fill[blue!40!white] (8,5) rectangle (9,6);

            \fill[blue!40!white] (9,0) rectangle (10,1);
            \fill[blue!40!white] (9,1) rectangle (10,2);
            \fill[blue!40!white] (9,2) rectangle (10,3);
            \fill[blue!40!white] (9,3) rectangle (10,4);
            \fill[blue!40!white] (9,4) rectangle (10,5);
            \fill[blue!40!white] (9,5) rectangle (10,6);
            
			\fill[blue!40!white] (10,0) rectangle (11,1);
            \draw[step=1cm,black,very thin] (3, 0) grid (11,6);
        \end{tikzpicture}
  \caption{Rake pattern only.}
\end{subfigure}%
\begin{subfigure}{.5\textwidth}
  \centering
  \begin{tikzpicture}[scale = 0.5]
            \draw[step=1cm,black,very thin] (2, 0) grid (10,6);
            \fill[blue!40!white] (2,0) rectangle (3,1);
            \fill[blue!40!white] (2,1) rectangle (3,2);
            \fill[blue!40!white] (2,2) rectangle (3,3);
            \fill[blue!40!white] (2,3) rectangle (3,4);
            \fill[blue!40!white] (2,4) rectangle (3,5);
            \fill[blue!40!white] (2,5) rectangle (3,6);

            \fill[blue!40!white] (3,1) rectangle (4,2);
            \fill[blue!40!white] (3,3) rectangle (4,4);
            \fill[blue!40!white] (3,5) rectangle (4,6);
            
            \fill[blue!40!white] (4,1) rectangle (5,2);
            \fill[blue!40!white] (4,3) rectangle (5,4);
            \fill[blue!40!white] (4,5) rectangle (5,6);

            \fill[blue!40!white] (5,1) rectangle (6,2);
            \fill[blue!40!white] (5,3) rectangle (6,4);
            \fill[blue!40!white] (5,5) rectangle (6,6);
            
            \fill[blue!40!white] (6,1) rectangle (7,2);
            \fill[blue!40!white] (6,3) rectangle (7,4);
            \fill[blue!40!white] (6,5) rectangle (7,6);
			
            \fill[blue!40!white] (7,1) rectangle (8,2);
            \fill[blue!40!white] (7,3) rectangle (8,4);
            \fill[blue!40!white] (7,5) rectangle (8,6);
            
            \fill[blue!40!white] (8,1) rectangle (9,2);
            \fill[blue!40!white] (8,3) rectangle (9,4);
            \fill[blue!40!white] (8,5) rectangle (9,6);

            \fill[blue!40!white] (9,0) rectangle (10,1);
            \fill[blue!40!white] (9,1) rectangle (10,2);
            \fill[blue!40!white] (9,2) rectangle (10,3);
            \fill[blue!40!white] (9,3) rectangle (10,4);
            \fill[blue!40!white] (9,4) rectangle (10,5);
            \fill[blue!40!white] (9,5) rectangle (10,6);

            \draw[step=1cm,black,very thin] (2, 0) grid (10,6);
        \end{tikzpicture}
  \caption{Stripe pattern only.}
\end{subfigure}
\begin{subfigure}{.5\textwidth}
  \centering
   \begin{tikzpicture}[scale = 0.5]
            \draw[step=1cm,black,very thin] (3, 0) grid (11,6);

            \fill[blue!40!white] (3,0) rectangle (4,1);
            \fill[blue!40!white] (3,1) rectangle (4,2);

            \fill[blue!40!white] (4,1) rectangle (5,2);
			\fill[blue!40!white] (4,2) rectangle (5,3);
			\fill[blue!40!white] (4,3) rectangle (5,4);
			\fill[blue!40!white] (4,4) rectangle (5,5);
			\fill[blue!40!white] (4,5) rectangle (5,6);

            \fill[blue!40!white] (5,1) rectangle (6,2);
			\fill[blue!40!white] (5,2) rectangle (6,3);
			\fill[blue!40!white] (5,3) rectangle (6,4);
			\fill[blue!40!white] (5,4) rectangle (6,5);
			\fill[blue!40!white] (5,5) rectangle (6,6);
			
			\fill[blue!40!white] (6,1) rectangle (7,2);
			\fill[blue!40!white] (7,1) rectangle (8,2);

			\fill[blue!40!white] (8,1) rectangle (9,2);
			\fill[blue!40!white] (8,2) rectangle (9,3);
            \fill[blue!40!white] (8,3) rectangle (9,4);
            \fill[blue!40!white] (8,4) rectangle (9,5);
            \fill[blue!40!white] (8,5) rectangle (9,6);

            \fill[blue!40!white] (9,1) rectangle (10,2);
            \fill[blue!40!white] (9,2) rectangle (10,3);
            \fill[blue!40!white] (9,3) rectangle (10,4);
            \fill[blue!40!white] (9,4) rectangle (10,5);
            \fill[blue!40!white] (9,5) rectangle (10,6);
            
			\fill[blue!40!white] (10,0) rectangle (11,1);
			\fill[blue!40!white] (10,1) rectangle (11,2);
            \draw[step=1cm,black,very thin] (3, 0) grid (11,6);
        \end{tikzpicture}
	\caption{The combination of the two patterns.}
	\end{subfigure}
	\caption{Consider the tract of land with dimensions $6 \times 8$. While
		the configuration with the rake pattern exhibits the occupancy of $28$, and the stripe pattern exhibits the occupancy of $30$, the combination,
		on the other hand, exhibits the occupancy of $26$.}\label{fig:comb_rake_stripe}
\end{figure}
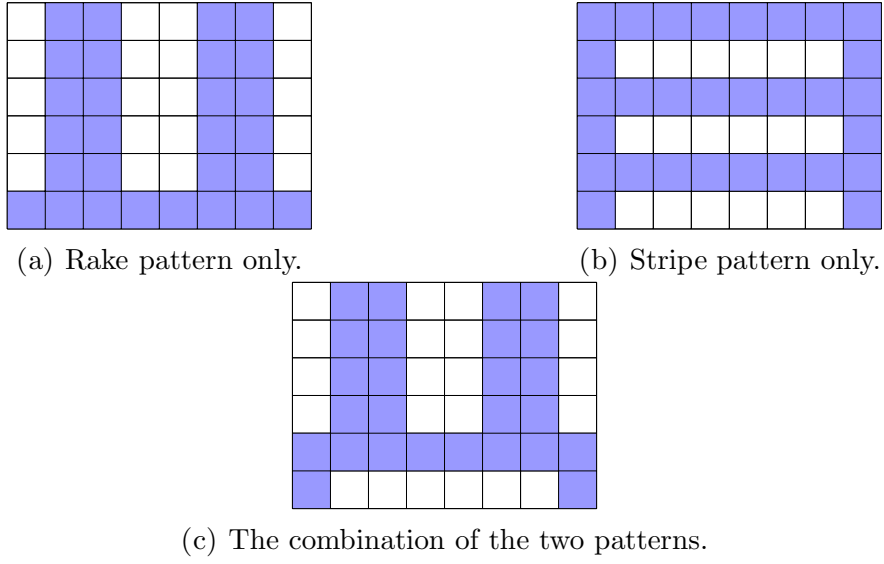
\begin{definition}[Rake--stripe pattern function]
	The function that gives us the occupancy $|C|$ of a configuration with the rake--stripe pattern is
\begin{equation}
\label{eq:formula_rakestripe}
            \mathcal{O}^{\rm rake-stripe}(m,n) := \begin{cases}
            n + 2 + (m-2)\left(\frac{1}{2}n\right), & \text{if } n \equiv 0 \Mod4; \\       
            n + 2 + (m-2)\left(\frac{1}{2}(n-1) + 1\right), &  \text{if } n \equiv 1 \Mod4; \\
            n +2 + (m-2)\left(\frac{1}{2}(n-2) + 2\right), & \text{if } n \equiv 2 \Mod4; \\
            n +2 +(m-2)\left(\frac{1}{2}(n-3) + 2\right), & \text{if } n \equiv 3 \Mod4;
            \end{cases}
        \end{equation}
        when $m,n\ge 2$.
\end{definition}    
    
\subsubsection{Check pattern}\label{subsubsec:check}
Check pattern is a periodic pattern depicted in Figure \ref{fig:checkfigure}. The maximal configuration $C^{\rm check}$ with the check pattern exhibits the building density of $1/2$.
\begin{figure}
        \centering
        \begin{tikzpicture}[scale = 0.5]
            \draw[step=1cm,black,very thin] (-0.9, -0.9) grid (11.9,4.9);
            
            \fill[blue!40!white] (-0.9,-0.9) rectangle (0,0);
            \fill[blue!40!white] (-0.9,1) rectangle (0,2);
            \fill[blue!40!white] (-0.9,3) rectangle (0,4);

            \fill[blue!40!white] (0,0) rectangle (1,1);
            \fill[blue!40!white] (0,2) rectangle (1,3);
            \fill[blue!40!white] (0,4) rectangle (1,4.9);

            \fill[blue!40!white] (1,-0.9) rectangle (2,0);
            \fill[blue!40!white] (1,1) rectangle (2,2);
            \fill[blue!40!white] (1,3) rectangle (2,4);

            \fill[blue!40!white] (2,0) rectangle (3,1);
            \fill[blue!40!white] (2,2) rectangle (3,3);
            \fill[blue!40!white] (2,4) rectangle (3,4.9);

            \fill[blue!40!white] (3,-0.9) rectangle (4,0);
            \fill[blue!40!white] (3,1) rectangle (4,2);
            \fill[blue!40!white] (3,3) rectangle (4,4);

            \fill[blue!40!white] (4,0) rectangle (5,1);
            \fill[blue!40!white] (4,2) rectangle (5,3);
            \fill[blue!40!white] (4,4) rectangle (5,4.9);

            \fill[blue!40!white] (5,-0.9) rectangle (6,0);
            \fill[blue!40!white] (5,1) rectangle (6,2);
            \fill[blue!40!white] (5,3) rectangle (6,4);

            \fill[blue!40!white] (6,0) rectangle (7,1);
            \fill[blue!40!white] (6,2) rectangle (7,3);
            \fill[blue!40!white] (6,4) rectangle (7,4.9);

            \fill[blue!40!white] (7,-0.9) rectangle (8,0);
            \fill[blue!40!white] (7,1) rectangle (8,2);
            \fill[blue!40!white] (7,3) rectangle (8,4);

            \fill[blue!40!white] (8,0) rectangle (9,1);
            \fill[blue!40!white] (8,2) rectangle (9,3);
            \fill[blue!40!white] (8,4) rectangle (9,4.9);

            \fill[blue!40!white] (9,-0.9) rectangle (10,0);
            \fill[blue!40!white] (9,1) rectangle (10,2);
            \fill[blue!40!white] (9,3) rectangle (10,4);

            \fill[blue!40!white] (10,0) rectangle (11,1);
            \fill[blue!40!white] (10,2) rectangle (11,3);
            \fill[blue!40!white] (10,4) rectangle (11,4.9);

            \fill[blue!40!white] (11,-0.9) rectangle (11.9,0);
            \fill[blue!40!white] (11,1) rectangle (11.9,2);
            \fill[blue!40!white] (11,3) rectangle (11.9,4);
            
            \draw[step=1cm,black,very thin] (-0.9, -0.9) grid (11.9,4.9);
        \end{tikzpicture}
        \caption{Check pattern.}\label{fig:checkfigure}
    \end{figure}
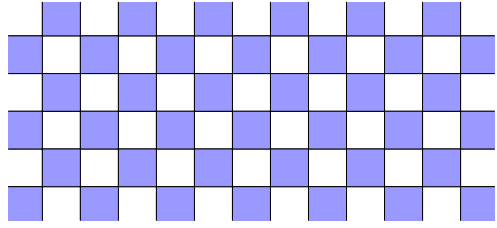   
    Note that the configuration on $\mathbb{Z}^2$ with the check pattern has the following two properties:
	\begin{itemize}
	    \item Each empty lot provides light to exactly three neighbors.
	    \item Each occupied lot receives light from exactly three neighbors.
	\end{itemize}

The finite maximal configurations with the check pattern are obtained by restricting $C^{\rm check}$ to a finite grid and building additional houses on the newly available lots. Among all such restrictions, we choose one that, in the end, yields the lowest occupancy, see Figure \ref{fig:checkpatternexamples}.

    \begin{definition}[Check pattern function]
        The function that gives us the occupancy $|C|$ of a configuration with the check pattern is
        \begin{equation*}
            \mathcal{O}^{\rm check}(m,n) := 
            2(m-1) + n + \FLOOR{\frac{m-1}{2}}\CEIL{\frac{n-2}{2}} + \CEIL{\frac{m-1}{2}}\FLOOR{\frac{n-2}{2}}.
        \end{equation*}
    \end{definition} 
    
        \begin{figure}
\centering
\begin{subfigure}{.5\textwidth}
  \centering
  \begin{tikzpicture}[scale = 0.5]
            \draw[step=1cm,black,very thin] (-1, 0) grid (10,4);
            
            \fill[blue!40!white] (-1,0) rectangle (0,1);
            \fill[blue!40!white] (-1,1) rectangle (0,2);
            \fill[blue!40!white] (-1,2) rectangle (0,3);
            \fill[blue!40!white] (-1,3) rectangle (0,4);

            \fill[blue!40!white] (0,0) rectangle (1,1);
            \fill[blue!40!white] (0,2) rectangle (1,3);

            \fill[blue!40!white] (1,0) rectangle (2,1);
            \fill[blue!40!white] (1,1) rectangle (2,2);
            \fill[blue!40!white] (1,3) rectangle (2,4);

            \fill[blue!40!white] (2,0) rectangle (3,1);
            \fill[blue!40!white] (2,2) rectangle (3,3);

            \fill[blue!40!white] (3,0) rectangle (4,1);
            \fill[blue!40!white] (3,1) rectangle (4,2);
            \fill[blue!40!white] (3,3) rectangle (4,4);

            \fill[blue!40!white] (4,0) rectangle (5,1);
            \fill[blue!40!white] (4,2) rectangle (5,3);

           \fill[blue!40!white] (5,0) rectangle (6,1);
            \fill[blue!40!white] (5,1) rectangle (6,2);
            \fill[blue!40!white] (5,3) rectangle (6,4);

            \fill[blue!40!white] (6,0) rectangle (7,1);
            \fill[blue!40!white] (6,2) rectangle (7,3);

			\fill[blue!40!white] (7,0) rectangle (8,1);
            \fill[blue!40!white] (7,1) rectangle (8,2);
            \fill[blue!40!white] (7,3) rectangle (8,4);

            \fill[blue!40!white] (8,0) rectangle (9,1);
            \fill[blue!40!white] (8,2) rectangle (9,3);

            \fill[blue!40!white] (9,0) rectangle (10,1);
            \fill[blue!40!white] (9,1) rectangle (10,2);
            \fill[blue!40!white] (9,2) rectangle (10,3);
            \fill[blue!40!white] (9,3) rectangle (10,4);

            \draw[step=1cm,black,very thin] (-1, 0) grid (10,4);
        \end{tikzpicture}
		\caption{}
\end{subfigure}%
\begin{subfigure}{.5\textwidth}
  \centering
  \begin{tikzpicture}[scale = 0.5]
            \draw[step=1cm,black,very thin] (-1, -1) grid (5,8);
            \fill[blue!40!white] (-1,-1) rectangle (0,0);
            \fill[blue!40!white] (-1,0) rectangle (0,1);
            \fill[blue!40!white] (-1,1) rectangle (0,2);
            \fill[blue!40!white] (-1,2) rectangle (0,3);
            \fill[blue!40!white] (-1,3) rectangle (0,4);
            \fill[blue!40!white] (-1,4) rectangle (0,5);
            \fill[blue!40!white] (-1,5) rectangle (0,6);
            \fill[blue!40!white] (-1,6) rectangle (0,7);
            \fill[blue!40!white] (-1,7) rectangle (0,8);

            \fill[blue!40!white] (0,-1) rectangle (1,0);
            \fill[blue!40!white] (0,0) rectangle (1,1);
            \fill[blue!40!white] (0,2) rectangle (1,3);
            \fill[blue!40!white] (0,4) rectangle (1,5);
            \fill[blue!40!white] (0,6) rectangle (1,7);

            \fill[blue!40!white] (1,-1) rectangle (2,0);
            \fill[blue!40!white] (1,1) rectangle (2,2);
            \fill[blue!40!white] (1,3) rectangle (2,4);
            \fill[blue!40!white] (1,5) rectangle (2,6);
            \fill[blue!40!white] (1,7) rectangle (2,8);

            \fill[blue!40!white] (2,-1) rectangle (3,0);
            \fill[blue!40!white] (2,0) rectangle (3,1);
            \fill[blue!40!white] (2,2) rectangle (3,3);
            \fill[blue!40!white] (2,4) rectangle (3,5);
            \fill[blue!40!white] (2,6) rectangle (3,7);

            \fill[blue!40!white] (3,-1) rectangle (4,0);
            \fill[blue!40!white] (3,1) rectangle (4,2);
            \fill[blue!40!white] (3,3) rectangle (4,4);
            \fill[blue!40!white] (3,5) rectangle (4,6);
            \fill[blue!40!white] (3,7) rectangle (4,8);

            \fill[blue!40!white] (4,-1) rectangle (5,0);
            \fill[blue!40!white] (4,0) rectangle (5,1);
            \fill[blue!40!white] (4,1) rectangle (5,2);
            \fill[blue!40!white] (4,2) rectangle (5,3);
            \fill[blue!40!white] (4,3) rectangle (5,4);
            \fill[blue!40!white] (4,4) rectangle (5,5);
            \fill[blue!40!white] (4,5) rectangle (5,6);
            \fill[blue!40!white] (4,6) rectangle (5,7);
            \fill[blue!40!white] (4,7) rectangle (5,8);

            \draw[step=1cm,black,very thin] (-1, -1) grid (5,8);
        \end{tikzpicture}
        \caption{}
	\end{subfigure}%
	\caption{Finite configurations obtained from the check pattern.}\label{fig:checkpatternexamples}
\end{figure}
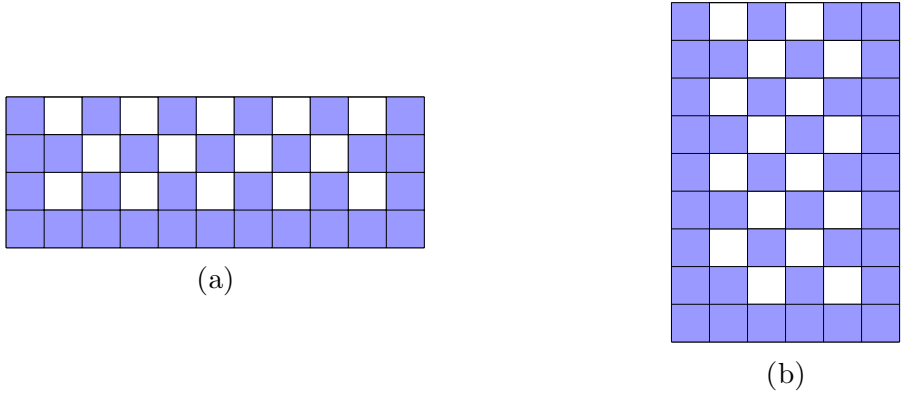

\section{\texorpdfstring{$I_{m,n}$}{I{m,n}} and bounds on \texorpdfstring{$E_{m,n}$}{E{m,n}}}\label{sec:in_eficientformula}
Recall that $I_{m,n}$ and $E_{m,n}$ are, respectively, the lowest and the highest occupancy attained among all of the maximal configurations on the $m\times n$ grid. In Lemma \ref{lm:crudeBound} we have proved
\begin{equation*}
\frac{1}{2}mn \le I_{m,n} \le E_{m,n}\le  \frac{3}{4}mn+\frac{m-1}{2}+\frac{n}{4}.
\end{equation*}
Furthermore, the occupancy of any concrete maximal configuration $C$ provides an upper bound on $I_{m,n}$ and a lower bound on $E_{m,n}$, since
$$I_{m,n} \le |C| \le E_{m,n}, \text{ for any maximal configuration } C.$$

In this section we obtain an explicit formula for $I_{m,n}$ (Theorem \ref{tm:Imn}) and provide improved bounds on $E_{m,n}$.

\subsection{Explicit formula for \texorpdfstring{$I_{m,n}$}{I{m,n}}}

The main result of this section is Theorem \ref{tm:Imn}. First we prove several auxiliary lemmas.
\begin{lemma}\label{lm:n+2}
	If $C$ is any maximal configuration on the $m\times n$ grid, $m,n\ge 2$, then the number of occupied lots in the two southernmost rows is at least $n+2$.
\end{lemma}
\begin{proof}
	We argue by contradiction. Suppose that there exist a maximal configuration $C$ containing at most $n+1$ occupied lots in rows $m-1$ and $m$. Note that the lot $(m,1)$ is occupied in any maximal configuration, as its only neighbors $(m,2)$ and $(m-1,1)$ always get sunlight as they are on the border. For the same reason, the lot $(m,n)$ is also occupied in any maximal configuration.
	
	We may also assume that the lot $(m-1,1)$ is occupied in $C$. If that was not the case, it would mean that $(m-1,2)$, $(m-1,3)$ and $(m,2)$ are all occupied because of the maximality of $C$. We could then move the house from $(m,2)$ to $(m-1,1)$, and we would still have a maximal configuration, with the same number of occupied lots as before. By similar reasoning, we may also assume that the lot $(m-1,n)$ is occupied in $C$.
	
	Therefore, of the remaining $2n-4$ lots $\{m-1,m\}\times\{2,3,\dots,n-1\}$, at most $n-3$ are occupied. By the pigeonhole principle, there must exist $2\le i\le n-1$ such that both lots $(i,m-1)$ and $(i,m)$ are empty. But this cannot happen in a maximal configuration, as there are no obstructions for $(i,m)$ to be occupied --- a contradiction.
\end{proof}

\begin{lemma}\label{lm:border2}
	Let $C$ be any maximal configuration on the $m\times n$ grid, $m,n\ge 2$, and let $1\le l\le m$ be a positive integer. The number of occupied lots in the restriction of the configuration $C$ to the grid $S=\{1,2,\dots,l\}\times\{1,2\}$ is at least one half of all the lots in $S$, i.e.\ 
	$$|C\cap S|\ge \frac{1}{2}|S| = l.$$
	The same is true if one takes $S$ to be along the eastern border of the grid, instead of the western border, i.e.\ if $S=\{1,2,\dots,l\}\times\{n-1,n\}$.
\end{lemma}
\begin{proof}
	Let $1\le i\le l$ and consider the lots $(i,1)$ and $(i,2)$. It is not possible that both of these are empty in a maximal configuration, since the only reason for leaving $(i,1)$ empty is if $(i,2)$ was occupied and, further, $(i,1)$ was its only source of light. From here, the claim follows for the western border. Because of the mirror symmetry East$\leftrightarrow$West, the eastern border case also follows.
\end{proof}

\begin{lemma}\label{lm:border3}
	Let $C$ be any maximal configuration on the $m\times n$ grid, $m\ge 2$, $n\ge 3$, and let $1\le l\le m$ be a positive integer. The number of occupied lots in the restriction of the configuration $C$ to the grid $S=\{1,2,\dots,l\}\times\{1,2,3\}$ is at least two thirds of all the lots in $S$, i.e.\ 
	$$|C\cap S|\ge \frac{2}{3}|S| = 2l.$$
	The same is true if one takes $S$ to be along the eastern border of the grid, instead of the western border, i.e.\ if $S=\{1,2,\dots,l\}\times\{n-2,n-1,n\}$.
\end{lemma}
\begin{proof}
	Let $1\le i\le l$ and consider the lots $(i,1)$, $(i,2)$ and $(i,3)$. We will show that either two out of three of these lots must be occupied or, failing that, there must be another row above them (i.e.\ $i>1$), and taking into account the additional three lots in the row above $(i-1,1)$, $(i-1,2)$ and $(i-1,3)$, at least 4 out of the total of 6 lots must be occupied. From this, the statement of the lemma will follow.
	
	First, we assume that $(i,1)$ is empty. Arguing as in the proof of Lemma \ref{lm:border2}, since $C$ is maximal, the lot $(i,2)$ must be occupied, with $(i,1)$ being its only source of light. Therefore, $(i,3)$ must also be occupied and we get two out of three occupied lots.
	
	Otherwise, assume that $(i,1)$ is occupied. If $(i,2)$ or $(i,3)$ is also occupied, we are done, so assume that both $(i,2)$ and $(i,3)$ are empty. Since $C$ is maximal, the only reason for $(i,2)$ being empty is:
	\begin{itemize}
		\item either $(i,2)$ is the only source of light to at least one of its occupied neighbors: $(i,3)$, $(i-1,2)$, or $(i,1)$;
		\item or alternatively, its neighbors to the east, south and west: $(i,3)$, $(i+1,2)$, and $(i,1)$ are all occupied.
	\end{itemize}
	Since $(i,3)$ is assumed empty, and $(i,1)$ has light coming from west border, the only possibility is that $(i-1,2)$ is occupied with $(i,2)$ being its only source of light. This implies that $(i-1,1)$ and $(i-1,3)$ are also occupied. Hence, 4 out of 6 lots $\{i-1,i\}\times\{1,2,3\}$ are occupied. This completes the proof of the western border version of the lemma. Again, the eastern border case follows from the mirror symmetry East$\leftrightarrow$West.
\end{proof}

\begin{lemma}\label{lm:width4col}
	Let $C$ be any maximal configuration on the $m\times n$ grid, $m\ge 2$, $n\ge 4$, and let $1\le l\le m$, $1\le t\le n-3$ be positive integers. The number of occupied lots in the restriction of the configuration $C$ to the grid $S=\{1,2,\dots,l\}\times\{t,t+1,t+2,t+3\}$ is at least one half of all the lots in $S$, i.e.\ 
	$$|C\cap S|\ge \frac{1}{2}|S| = 2l.$$
\end{lemma}
\begin{proof}
	Let $1\le i\le l$ and consider the lots $(i,t)$, $(i,t+1)$, $(i,t+2)$ and $(i,t+3)$. We will show that either two out of four of these lots must be occupied or, failing that, there must be another row above them (i.e.\ $i>1$), and taking into account the additional four lots in the row above $(i-1,t)$, $(i-1,t+1)$, $(i-1,t+2)$ and $(i-1,t+3)$, at least four out of the total of eight lots must be occupied. From this, the statement of the lemma will follow.
	
	We may assume that at most one out of four lots $(i,t)$, $(i,t+1)$, $(i,t+2)$ and $(i,t+3)$ is occupied, as otherwise we are done. Let us first assume that either $(i,t)$ or $(i,t+1)$ is the only occupied lot among those four. Since $C$ is maximal, there are exactly four possible reasons for $(i,t+2)$ being empty:
	\begin{itemize}
		\item its neighbors to the east, south and west: $(i,t+3)$, $(i+1,t+2)$, and $(i,t+1)$ are all occupied;
		\item or $(i,t+2)$ is the only source of light to its western neighbor $(i,t+3)$;
		\item or $(i,t+2)$ is the only source of light to its eastern neighbor $(i,t+1)$;
		\item or $(i,t+2)$ is the only source of light to its northern neighbor $(i-1,t+2)$.
	\end{itemize}

	In both cases (whether $(i,t)$ or $(i,t+1)$ is the only occupied lot among $(i,t)$, $(i,t+1)$, $(i,t+2)$ and $(i,t+3)$) the only possibility is that $(i-1,t+2)$ is occupied with $(i,t+2)$ being its only source of light. This implies that $(i-1,t+1)$ and $(i-1,t+3)$ are also occupied. Hence, at least four out of eight lots $\{i-1,i\}\times\{t,t+1,t+2,t+3\}$ are occupied.
	
	A similar argument shows that if $(i,t+2)$ or $(i,t+3)$ is the only occupied lot among $(i,t)$, $(i,t+1)$, $(i,t+2)$ and $(i,t+3)$; then all of the lots $(i-1,t)$, $(i-1,t+1)$ and $(i-1,t+2)$ must also be occupied, and therefore again, at least four out of eight lots $\{i-1,i\}\times\{t,t+1,t+2,t+3\}$ are occupied.
	
	Finally, assume that none of the lots $(i,t)$, $(i,t+1)$, $(i,t+2)$ and $(i,t+3)$ are occupied. Arguing as in the previous cases, we conclude that all the lots: $(i-1,t)$, $(i-1,t+1)$, $(i-1,t+2)$ and $(i-1,t+3)$ must be occupied. Hence, once more, at least four out of eight lots $\{i-1,i\}\times\{t,t+1,t+2,t+3\}$ are occupied.
\end{proof}

With this lemmas, we can improve the lower bound on the number of occupied lots in any maximal configuration and compute, exactly, the occupancies of inefficient configurations $I_{m,n}$.
\begin{theorem}[Sharp lower bound on $I_{m,n}$]\label{tm:Imn}
	If $C$ is any maximal configuration on the $m\times n$ grid, $m,n\ge 2$, then
	$$|C| \ge I_{m,n} \ge \begin{cases}
	\frac{mn}{2}+2, &\text{if } n\equiv 0\Mod{4},\\
	\frac{m(n+2)}{2}, &\text{if } n\equiv 2\Mod{4},\\
	\frac{m(n+1)}{2}+1, &\text{if } n\equiv 1\Mod{2}.\\
	\end{cases}$$
\end{theorem}
\begin{remark}
	These bounds are sharp since they are attained by configurations with the rake--stripe pattern combination introduced in Section \ref{subsect:rakeStripe}. The expressions above match the occupancy of the rake--stripe configurations given in equation \eqref{eq:formula_rakestripe} and are, therefore, the values of $I_{m,n}$.
\end{remark}
\begin{proof}[Proof of Theorem \ref{tm:Imn}]
	If $n=2$, all the lots must be occupied and we see that the bound holds. We now assume $n\ge3$, and consider four cases depending on the remainder after dividing $n$ by $4$. In all four cases, we split the grid in the \emph{lower part} consisting of the two southernmost rows $L=\{m-1,m\}\times[n]$ and the \emph{upper part} consisting of the rest $U=[m-2]\times[n]$. The upper part we further divide depending on the said remainder. Note that, from Lemma \ref{lm:n+2}, $|C\cap L|\ge n+2$.
	
	\emph{Case 1.} $n\equiv 0\Mod{4}$ We divide $U$ in equal blocks of width $4$ and height $m-2$ (Figure \ref{fig:ICase1}). Applying Lemma \ref{lm:width4col} to each block, we get $|C\cap U|\ge \frac{1}{2}|U|=\frac{(m-2)n}{2}$. Thus, $|C|=|C\cap U|+|C\cap L|\ge \frac{mn}{2}+2$.
	
	\emph{Case 2.} $n\equiv 1\Mod{4}$ We divide $U$ in the western block of width $3$, the eastern block of width $2$, and the middle portion (if any) into equal blocks of width $4$ (Figure \ref{fig:ICase2}). Applying Lemma \ref{lm:border3} to the western block, Lemma \ref{lm:border2} to the eastern block, and Lemma \ref{lm:width4col} to the blocks in the middle portion, we get $|C\cap U|\ge \frac{2}{3}(m-2)\cdot3 + \frac{1}{2}(m-2)(n-3)=\frac{(m-2)(n+1)}{2}$. Thus, $|C|=|C\cap U|+|C\cap L|\ge \frac{m(n+1)}{2}+1$.
	
	\emph{Case 3.} $n\equiv 2\Mod{4}$ We divide $U$ in the western block of width $3$, the eastern block of width $3$, and the middle portion (if any) into equal blocks of width $4$ (Figure \ref{fig:ICase3}). Applying Lemma \ref{lm:border3} to the western and eastern block, and Lemma \ref{lm:width4col} to the blocks in the middle portion, we get $|C\cap U|\ge \frac{2}{3}(m-2)\cdot6 + \frac{1}{2}(m-2)(n-6)=\frac{(m-2)(n+2)}{2}$. Thus, $|C|=|C\cap U|+|C\cap L|\ge \frac{m(n+2)}{2}$.
	
	\emph{Case 4.} $n\equiv 3\Mod{4}$ We divide $U$ in the western block of width $3$, and the rest (if any) into equal blocks of width $4$ (Figure \ref{fig:ICase4}). Applying Lemma \ref{lm:border3} to the western block, and Lemma \ref{lm:width4col} to the blocks of width $4$, we get $|C\cap U|\ge \frac{2}{3}(m-2)\cdot3 + \frac{1}{2}(m-2)(n-3)=\frac{(m-2)(n+1)}{2}$. Thus, $|C|=|C\cap U|+|C\cap L|\ge \frac{m(n+1)}{2}+1$.
\end{proof}

\begin{figure}
	\begin{subfigure}{0.49\textwidth}\centering
		\begin{tikzpicture}[scale=0.3,trim left = (A), trim right=(B)]
		\node at (0,0) (A) {};
		\node at (20,0) (B) {};
		\draw (0,0) rectangle (20,2) node[pos=.5] {\tiny$2\times n$};
		\draw (0,2) rectangle (4,8);
		\draw[<->](0.2,5) -- node[above]  {\tiny$4$} (3.8,5);
		\draw (4,2) rectangle (8,8);
		\draw[<->](4.2,5) -- node[above]  {\tiny$4$} (7.8,5);
		\draw (8,2) rectangle (12,8);
		\node at (10,5) {$\dots$};
		\draw (12,2) rectangle (16,8);
		\draw[<->](12.2,5) -- node[above] {\tiny$4$} (15.8,5);
		\draw (16,2) rectangle (20,8);
		\draw[<->](16.2,5) -- node[above] {\tiny$4$} (19.8,5);
		
		\draw[<->](20.8,2.2) -- node[rotate=-90,above]  {\tiny$m-2$} (20.8,7.8);
		\end{tikzpicture}
		\caption{$n\equiv 0\Mod{4}$}\label{fig:ICase1}
	\end{subfigure}
	\begin{subfigure}{0.49\textwidth}\centering
		\begin{tikzpicture}[scale=0.3,trim left = (A), trim right=(B)]
		\node at (0,0) (A) {};
		\node at (17,0) (B) {};
		\draw (0,0) rectangle (17,2) node[pos=.5] {\tiny$2\times n$};
		\draw (0,2) rectangle (3,8);
		\draw[<->](0.2,5) -- node[above]  {\tiny$3$} (2.8,5);
		\draw (3,2) rectangle (7,8);
		\draw[<->](3.2,5) -- node[above]  {\tiny$4$} (6.8,5);
		\draw (7,2) rectangle (11,8);
		\node at (9,5) {$\dots$};
		\draw (11,2) rectangle (15,8);
		\draw[<->](11.2,5) -- node[above] {\tiny$4$} (14.8,5);
		\draw (15,2) rectangle (17,8);
		\draw[<->](15.2,5) -- node[above] {\tiny$2$} (16.8,5);
		
		\draw[<->](17.8,2.2) -- node[rotate=-90,above]  {\tiny$m-2$} (17.8,7.8);
		\end{tikzpicture}
		\caption{$n\equiv 1\Mod{4}$}\label{fig:ICase2}
	\end{subfigure}
	\begin{subfigure}{0.49\textwidth}\centering
		\begin{tikzpicture}[scale=0.3,trim left = (A), trim right=(B)]
		\node at (0,0) (A) {};
		\node at (18,0) (B) {};
		\draw (0,0) rectangle (18,2) node[pos=.5] {\tiny$2\times n$};
		\draw (0,2) rectangle (3,8);
		\draw[<->](0.2,5) -- node[above]  {\tiny$3$} (2.8,5);
		\draw (3,2) rectangle (7,8);
		\draw[<->](3.2,5) -- node[above]  {\tiny$4$} (6.8,5);
		\draw (7,2) rectangle (11,8);
		\node at (9,5) {$\dots$};
		\draw (11,2) rectangle (15,8);
		\draw[<->](11.2,5) -- node[above] {\tiny$4$} (14.8,5);
		\draw (15,2) rectangle (18,8);
		\draw[<->](15.2,5) -- node[above] {\tiny$3$} (17.8,5);
		
		\draw[<->](18.8,2.2) -- node[rotate=-90,above]  {\tiny$m-2$} (18.8,7.8);
		\end{tikzpicture}
		\caption{$n\equiv 2\Mod{4}$}\label{fig:ICase3}
	\end{subfigure}
	\begin{subfigure}{0.49\textwidth}\centering
		\begin{tikzpicture}[scale=0.3,trim left = (A), trim right=(B)]
		\node at (0,0) (A) {};
		\node at (19,0) (B) {};
		\draw (0,0) rectangle (19,2) node[pos=.5] {\tiny$2\times n$};
		\draw (0,2) rectangle (3,8);
		\draw[<->](0.2,5) -- node[above]  {\tiny$3$} (2.8,5);
		\draw (3,2) rectangle (7,8);
		\draw[<->](3.2,5) -- node[above]  {\tiny$4$} (6.8,5);
		\draw (7,2) rectangle (11,8);
		\node at (9,5) {$\dots$};
		\draw (11,2) rectangle (15,8);
		\draw[<->](11.2,5) -- node[above] {\tiny$4$} (14.8,5);
		\draw (15,2) rectangle (19,8);
		\draw[<->](15.2,5) -- node[above] {\tiny$4$} (18.8,5);
		
		\draw[<->](19.8,2.2) -- node[rotate=-90,above]  {\tiny$m-2$} (19.8,7.8);
		\end{tikzpicture}
		\caption{$n\equiv 3\Mod{4}$}\label{fig:ICase4}
	\end{subfigure}
	\caption{The subdivisions of the $m\times n$ grid used in the proof of Theorem \ref{tm:Imn}.}
\end{figure}
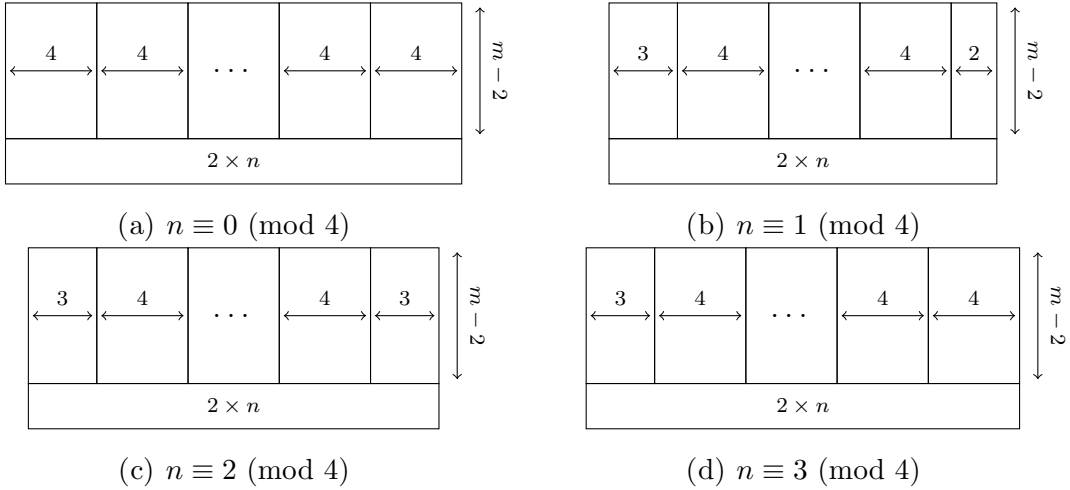

\subsection{Improved bounds for \texorpdfstring{$E_{m,n}$}{E{m,n}}}

We now turn to the efficient configurations and improve bounds on $E_{m,n}$.

\begin{proposition}[Upper bound on $E_{m,n}$] \label{prop:upperBoundEmn}
	If $C$ is any maximal configuration on the $m\times n$ grid, $m,n\ge 2$, then
	\begin{equation*}
	|C| \le E_{m,n} \leq  \begin{cases}
	mn - \FLOOR{\frac{n}{4}} \cdot (m - 1), & \text{if } n \not\equiv 3 \Mod4, \\
	mn - \FLOOR{\frac{n}{4}} \cdot (m - 1) - \FLOOR{\frac{m}{2}}, & \text{if } n \equiv 3 \Mod4.
	\end{cases}
	\end{equation*}
\end{proposition}
\begin{proof}
	The idea for this proof is taken from the solution of the problem posed in the 10\textsuperscript{th} Middle European Mathematical Olympiad (see \cite{MEMO}). We begin the proof by dividing the tract of land, not including the bottom row, into $\FLOOR{\frac{n}{4}} \cdot (m - 1)$ adjacent blocks of size $1 \times 4$, starting from the west. Now we claim that there exist an injection from the set of these blocks to the set of empty lots. Consider a single such block. Two situations are possible. Either this block itself contains an empty lot, in which case it can be mapped to its easternmost empty lot, or it is entirely occupied, which implies that the block directly below it now must have at least two empty lots. In that situation, we map the upper block to the westernmost empty lot of the lower block, and the lower block to its easternmost empty lot. This mapping is clearly an injection. In the case $n \equiv 3 \Mod4$, we can extend our injection by dividing the easternmost three columns into $\FLOOR{\frac{m}{2}}$ blocks of size $2\times 3$. Since each of these contains at least one empty lot, we can injectively map those blocks into empty lots as well. This completes the proof.
\end{proof}
\begin{remark}
	Note that with the above proposition and the brick pattern function defined in Section \ref{subsubsect:brickPattern}, we have completely solved the case $m\times 3$, $m\geq 2$. We have
	$$ 2m + \CEIL{\frac{m}{2}} \leq E_{m,3} \leq 3m - \FLOOR{\frac{m}{2}}. $$
\end{remark}

Before proceeding, we recall some well-known facts about the floor and ceiling functions which will be useful latter on:
\begin{gather}
\forall x\in\bbR \quad \CEIL{-x} = -\FLOOR{x},\label{eq:CF1}\\
\forall x\in\bbR \quad \FLOOR{-x} = -\CEIL{x},\label{eq:CF2}\\
\forall x\in\bbR, \forall k\in\bbZ \quad \CEIL{x+k} = \CEIL{x}+k,\label{eq:CF3}\\
\forall x\in\bbR, \forall k\in\bbZ \quad \FLOOR{x+k} = \FLOOR{x}+k,\label{eq:CF4}\\
\forall k,l\in\bbN \quad \FLOOR{\frac{k}{l}} = \CEIL{\frac{k-l+1}{l}}.\label{eq:CF5}
\end{gather}

We will now give an improved upper bound on the size of maximal configurations. The bound is not going to be explicit, but given by the following recurrence relation. For a fixed $n\in\bbN$, we define the sequence $(R_{m,n})_{m\in\bbN\cup\{0\}}$ as follows:
\begin{equation}\label{eq:recurrR}
\begin{cases}
R_{0,n}=0, \quad R_{1,n}=n,\\
R_{m,n}=R_{m-1,n}+n-\FLOOR{\frac{R_{m-1,n}-R_{m-2,n}}{3}}, & \text{if } m\ge 2.\\
\end{cases}
\end{equation}

\begin{remark}
	From \eqref{eq:recurrR} it is not hard to see that
	\begin{multline*}
	R_{m,n} = mn-\underbrace{\FLOOR{\frac{n}{3}}-\FLOOR{\frac{n-\FLOOR{\frac{n}{3}}}{3}} - \dots - \FLOOR{\frac{n-\FLOOR{\frac{n-\FLOOR{\cdots}}{3}}}{3}}}_{m-1 \text{ terms}} =\\
	= mn - f_n(0) - f_n^{2}(0)-\dots -f_n^{m-1}(0),
	\end{multline*}
	where $f_n(x) = \FLOOR{\dfrac{n-x}{3}}$, and $f_n^k$ denotes the composition $\underbrace{f_n\circ f_n \circ \dots \circ f_n}_{k \text{ times}}$.
\end{remark}

\begin{remark}
	Using \eqref{eq:CF1} and \eqref{eq:CF3} we can rewrite the recurrence relation \eqref{eq:recurrR} as
	\begin{multline}\label{eq:recurrMONOTONE}
	R_{m,n}=R_{m-1,n}+n-\FLOOR{\frac{R_{m-1,n}-R_{m-2,n}}{3}} =\\
	=n+\CEIL{R_{m-1,n}+\frac{-R_{m-1,n}+R_{m-2,n}}{3}} =n+\CEIL{\frac{2R_{m-1,n}+R_{m-2,n}}{3}},  \text{ for } m\ge 2,
	\end{multline}
	where we noted that $R_{m-1,n}$ is always an integer.
\end{remark}

We will also need the following lemma.
\begin{lemma}\label{lm:RowAbove}
	Let $m\ge 2$, $n\ge1$, and let $C$ be any permissible configuration on the $m\times n$ grid. If $2\le r\le m$ then
	$$|C\cap (\{r-1\}\times[n])|\le n- \FLOOR{\frac{|C\cap (\{r\}\times[n])|}{3}}.$$
	In other words, if there are $k$ occupied lots in row $r$, then there are at most $n-\FLOOR{\frac{k}{3}}$ occupied lots in the row above it.
\end{lemma}
\begin{proof}
	Let $(r,j)$ be an arbitrary occupied lot in row $r\ge 2$ of a permissible configuration $C$, which is not on the eastern or western border ($1<j<n$). It cannot happen that all three lots $(r-1,j-1)$, $(r-1,j)$, and $(r-1,j+1)$ in the row above are occupied, as this would mean that $(r-1,j)$ is blocked from the sun. We are, therefore, able to construct a map $\FJADEF{f}{C\cap(\{r\}\times\{2,\dots,n-1\})}{C^c\cap(\{r-1\}\times[n])}$, which, to each occupied lot $(r,j) \in C\cap(\{r\}\times\{2,\dots,n-1\})$, assigns an empty lot in $\{(r-1,j-1),(r-1,j),(r-1,j+1)\}$. By construction, each empty lot in row $r-1$ can be image of at most 3 occupied lots in row $r$, hence
	$$|C\cap(\{r\}\times\{2,\dots,n-1\})| \le 3|C^c\cap(\{r-1\}\times[n])|.$$
	Since $|C^c\cap(\{r-1\}\times[n])|=n-|C\cap(\{r-1\}\times[n])|$, and $|C\cap(\{r\}\times[n])|-2 \le |C\cap(\{r\}\times\{2,\dots,n-1\})|$, we get
	$$\frac{|C\cap(\{r\}\times[n])|-2}{3} \le n-|C\cap(\{r-1\}\times[n])|.$$
	Therefore,
	$$|C\cap(\{r-1\}\times[n])| \le n-\frac{|C\cap(\{r\}\times[n])|-2}{3}$$
	and since the number on the left hand side is an integer
	$$|C\cap(\{r-1\}\times[n])| \le \FLOOR{n-\frac{|C\cap(\{r\}\times[n])|-2}{3}}.$$
	
	It remains to observe that
	$$\FLOOR{n-\frac{|C\cap(\{r\}\times[n])|-2}{3}} = n - \CEIL{\frac{|C\cap(\{r\}\times[n])|-2}{3}} = n - \FLOOR{\frac{|C\cap(\{r\}\times[n])|}{3}},$$
	where we used \eqref{eq:CF4}, \eqref{eq:CF2}, and \eqref{eq:CF5} for $l=3$.
\end{proof}

\begin{theorem}[Improved upper bound on $E_{m,n}$]\label{tm:recurrBOUND}
	If $C$ is any maximal configuration on the $m\times n$ grid, $m,n\ge 1$, then
	$$|C|\le E_{m,n} \le R_{m,n},$$
	where $R_{m,n}$ is defined by the recurrence relation \eqref{eq:recurrR} (or \eqref{eq:recurrMONOTONE}).
\end{theorem}
\begin{remark}
	This bound seems to be sharp when $n=1,2,3,4,5\text{ or }8$, regardless of $m$ (see Table \ref{tablicica}), but it is not sharp in general (see Remark \ref{rem:selfimprove} below).
\end{remark}

\begin{proof}[Proof of Theorem \ref{tm:recurrBOUND}]
	Let the number of columns $n\in\bbN$ be fixed. We will argue by induction on the number of rows $m$.
	
	The statement of the theorem is trivially true for $m=0$ and $m=1$ (and any $n\in\bbN$). The claim for $m=0$ seems artificial, but it allows us for a simpler proof. To complete the inductive step, let us assume that $m\ge2$, that the theorem holds for $m-1$ and $m-2$, and that $C$ is a maximal configuration on $m\times n$ grid.
	
	We set $S=\{2,\dots,m\}\times[n]$ and $T=\{3,\dots,m\}\times[n]$ (if $m=2$ then $T=\emptyset$).
	Note that the number of occupied lots in $C$ in row 2 is $|C\cap S| - |C\cap T|$. By Lemma \ref{lm:RowAbove}, the number of occupied lots in row 1 is at most $n-\FLOOR{\frac{|C\cap S| - |C\cap T|}{3}}$. Hence,
	\begin{gather*}
		|C|\le n-\FLOOR{\frac{|C\cap S| - |C\cap T|}{3}} + |C\cap S|= n+\CEIL{\frac{|C\cap T|-|C\cap S|}{3}+|C\cap S|},\\
		|C|\le n+\CEIL{\frac{|C\cap T|+2|C\cap S|}{3}}.
	\end{gather*}
	Above we used \eqref{eq:CF1} and \eqref{eq:CF3}. Note that by assumption $|C\cap S|\le R_{m-1,n}$ and $|C\cap T|\le R_{m-2,n}$, as $|C\cap S|$ and $|C\cap T|$ are, perhaps not maximal but certainly, permissible configurations on grids of dimension $(m-1)\times n$ and $(m-2)\times n$, respectively. Therefore,
	$$|C| \le n + \CEIL{\frac{R_{m-2,n}+2R_{m-1,n}}{3}} = R_{m,n}.$$
	This completes the inductive step and the proof of the theorem.
\end{proof}

\begin{remark}\label{rem:selfimprove}
	The argument used in the proof of Theorem \ref{tm:recurrBOUND} actually shows
	$$E_{m,n} \le n + \CEIL{\frac{E_{m-2,n}+2E_{m-1,n}}{3}}, \text{ for } m\ge 2.$$
	Since the right-hand side of this expression is increasing in both $E_{m-2,n}$ and $E_{m-1,n}$, it is possible to obtain even better bounds on $E_{m,n}$ than those achieved by $R_{m,n}$ simply by computing explicitly (using e.g.\ exhaustive search) few strategically chosen values of $E_{m,n}$, and then letting the recurrence relation take over. The bound is, in a way, self-improving. As an example, in Table \ref{tab:RmnEmn} are listed values of $R_{m,n}$ and $E_{m,n}$, for $2\le m \le 10$ and $n=7$. The values $R_{m,n}$ are clearly not matching $E_{m,n}$, but if we change $R_{3,7}$ into $17$ and $R_{4,7}$ into $22$, and if we update all the remaining values $R_{m,n}$, $m\ge 5$, according to \eqref{eq:recurrMONOTONE}, we obtain the correct values $E_{m,7}$, for $5\le m\le 16$. We were able to check these using computer assisted exhaustive search.
\end{remark}

\begin{table}
	\begin{minipage}{.2\textwidth}
		\begin{tabular}{c|c}
		$m$ & $R_{m,7}$ \\
		\hline\hline 
		2 & 12 \\
		3 & \bf18 \\
		4 & \bf23 \\
		5 & 29 \\
		6 & 34 \\
		7 & 40 \\
		8 & 45 \\
		9 & 51 \\
		10 & 56 \\
		\vdots & \vdots
	\end{tabular}
	\end{minipage}
	\begin{minipage}{.2\textwidth}
		\begin{tabular}{c|c}
		$m$ & $E_{m,7}$ \\
		\hline\hline 
		2 & 12 \\
		3 & 17 \\
		4 & 22 \\
		5 & 28 \\
		6 & 33 \\
		7 & 39 \\
		8 & 44 \\
		9 & 50 \\
		10 & 55 \\
		\vdots & \vdots
		\end{tabular}
	\end{minipage}
	\caption{$R_{m,7}$ and $E_{m,7}$. Changing $R_{3,7}$ to $17$ and $R_{4,7}$ to $22$ and updating all the remaining values $R_{m,n}$, $m\ge 5$,  according to \eqref{eq:recurrMONOTONE}, gives the correct values $E_{m,7}$, for $5\le m\le 16$, which we were able to check using computer assisted exhaustive search.}\label{tab:RmnEmn}
\end{table}

\begin{remark}
	It is straightforward to check, by induction, that the bound on $E_{m,n}$, given by the recurrence relation $R_{m,n}$, in Theorem \ref{tm:recurrBOUND} is better than the one previously obtained in Proposition \ref{prop:upperBoundEmn}.
\end{remark}

\section{IP formulation and explicit solutions}\label{sec:ip}
In this section we describe integer programming (IP) formulations for the problems of finding efficient and inefficient maximal configurations. By solving these problems we gain insight in the shape of the explicit solutions and the associated occupancies. The decision variables consist of a $m \times n$ matrix $X \in \{0,1\}^{m \times n}$ representing the configuration $C$ on the tract of land with dimensions $m \times n$.  

For the sake of elegance of the formulations, we introduce the following notation: 
\begin{definition} Let $C$ be a configuration of houses on a tract of land with dimensions $m \times n$. 
For a lot $(i,j)$ we define the following propositions (if applicable):
    \begin{itemize}
        \item $\mathbb{P}_{i,j}^E:=$  ``$(i,j)$ is empty and it is the only source of light to its eastern immediate neighbor.'' $\iff$ $C_{i,j+1} + C_{i,j+2}+ C_{i+1,j+1} = 3$. \\
        \item $\mathbb{P}_{i,j}^W:=$ ``$(i,j)$ is empty and it is the only source of light to its western immediate neighbor.'' $\iff$  $C_{i,j-1} + C_{i,j-2}+ C_{i+1,j-1} = 3$. \\
        \item $\mathbb{P}_{i,j}^N:=$ ``$(i,j)$ is empty and it is the only source of light to its northern immediate neighbor.'' $\iff$ $C_{i-1,j+1} + C_{i-1,j-1}+ C_{i-1,j} = 3$. \\
        \item $\mathbb{P}_{i,j}^C:=$ ``$(i,j)$ is itself blocked from the sun.'' $\iff$ $C_{i,j+1} + C_{i,j-1}+ C_{i+1,j} = 3$. \\
    \end{itemize}
\end{definition}

\begin{remark}
 Consider a maximal configuration $C$. We stress here that there are two distinct reasons why there would occur $C_{i,j} = 0$
 \begin{itemize}
 \item either $(i,j)$ is itself blocked from the sun,
 \item or $(i,j)$ is the only source of light to one of its immediate neighbors. 
 \end{itemize}
\end{remark}
This can be summarized with the following characterization of maximal configurations pertaining on the algebraic formulas:
\begin{lemma}
\label{characterizationofmaximalconfs}
Let $C$ be a maximal configuration. For every $1<i<m$, $2<j<n-1$ we have: 
\begin{equation*}
 C_{i,j} = 0 \Rightarrow \, 
\mathbb{P}_{i,j}^E \vee \, \mathbb{P}_{i,j}^W \vee \mathbb{P}_{i,j}^N \vee \mathbb{P}_{i,j}^C.
\end{equation*}
with the similar formulas for the indices $i = 1,m$, $j = 1,2,n-1,n$ while keeping in mind the existence of its neighbors.
\end{lemma}

\noindent\textbf{Efficient configurations (IP formulation).}
Let $m,n \in \mathbb{N}$ be the dimensions of the tract of land. The IP formulation of the problem of finding efficient maximal configurations is the following:
\begin{equation}
\label{eq:efficientmip}
\begin{array}{ll}
\text{maximize}  & \displaystyle\sum\limits_{i=1}^{m}\displaystyle\sum\limits_{j=1}^{n} X_{i,j} \\
\text{subject to}& \neg \mathbb{P}_{i,j}^C,  \;\; i= 1,\dots m-1, \;\; j=2 ,\dots, n-1\\
    &  X_{i,j} \in \{0,1\}, \;\; i=1, \dots m, \;\; j=1 ,\dots, n
\end{array}
\end{equation}

\noindent\textbf{Inefficient configurations (IP formulation).}
Let $m,n \in \mathbb{N}$ be the dimensions of the tract of land. The IP formulation of the problem of finding inefficient maximal configurations is the following:
\begin{equation}
\label{eq:inefficientmip}
\begin{array}{ll}
\text{minimize}  & \displaystyle\sum\limits_{i=1}^{m}\displaystyle\sum\limits_{j=1}^{n} X_{i,j}\\
\text{subject to}&  X_{i,j} = 0 \Rightarrow 
\mathbb{P}_{i,j}^E \vee \, \mathbb{P}_{i,j}^W \vee \mathbb{P}_{i,j}^N \vee \mathbb{P}_{i,j}^C, \\
    &  X_{i,j} \in \{0,1\}, \;\; i=1, \dots m, \;\; j=1 ,\dots, n
\end{array}
\end{equation}

\begin{remark}
	The constraint $\neg \mathbb{P}_{i,j}^C$ in \eqref{eq:efficientmip} ensures that the resulting configuration is permissible, while the constraint $X_{i,j} = 0 \Rightarrow \mathbb{P}_{i,j}^E \vee \, \mathbb{P}_{i,j}^W \vee \mathbb{P}_{i,j}^N \vee \mathbb{P}_{i,j}^C$ in \eqref{eq:inefficientmip} ensures that the resulting configuration is maximal. Although we can put both constraints in both optimization problems, it turns out that it is sufficient to use only one of them in each.
\end{remark}

In order to rephrase the above logical constraints into the equivalent algebraic constraints according to the standard integer programming convention, one can consult \cite[Chapters 8 and 9]{Williams}. However, we have used IBM ILOG CPLEX for solving the above integer programs, and IBM Optimization Programming Language (OPL) supports the definitions of constraints in the form of logical constraints. The reference guide for integer programming with IBM ILOG CPLEX can be found in \cite{cplex2009v12}.

In Theorem \ref{tm:Imn} we have concluded that the rake--stripe pattern configurations are examples of inefficient configurations. However, we still do not know examples of efficient configurations for all $m,n \in  \mathbb{N}$. The only insight into the set of efficient configurations comes from the explicit solutions to the problem \eqref{eq:efficientmip}. We have gathered the computed occupancies of efficient configurations in Table \ref{tablicica}.

\begin{table}
    \begin{tabularx}{1.05\textwidth} { 
  | >{\centering\arraybackslash}X 
  || >{\centering\arraybackslash}X 
  | >{\centering\arraybackslash}X 
  | >{\centering\arraybackslash}X 
  | >{\centering\arraybackslash}X 
  | >{\centering\arraybackslash}X 
  | >{\centering\arraybackslash}X 
  | >{\centering\arraybackslash}X 
  | >{\centering\arraybackslash}X 
  | >{\centering\arraybackslash}X 
  | >{\centering\arraybackslash}X 
  | >{\centering\arraybackslash}X 
  | >{\centering\arraybackslash}X 
  | >{\centering\arraybackslash}X 
  | >{\centering\arraybackslash}X 
  | >{\centering\arraybackslash}X 
  | >{\centering\arraybackslash}X 
  | >{\raggedleft\arraybackslash}X | }
  \hline
\Small{m/n} & 2 & 3 & 4 & 5 & 6 & 7 & 8 & 9 & 10 & 11 & 12 & 13 & 14 & 15 & 16 \\
\hline
\hline
2&	\cellcolor{ForestGreen!60} 4&	\cellcolor{ForestGreen!60}5&	\cellcolor{ForestGreen!60}7&	\cellcolor{blue!25}9&	\cellcolor{ForestGreen!60}10&	\cellcolor{blue!25}12&	\cellcolor{blue!25}14&	\cellcolor{blue!25}15&	\cellcolor{blue!25}17&	\cellcolor{blue!25}19&	\cellcolor{blue!25}20&	\cellcolor{blue!25}22&	\cellcolor{blue!25}24&	\cellcolor{blue!25}25&	\cellcolor{blue!25}27\\
\hline
3&	\cellcolor{blue!25} 6&	\cellcolor{yellow!25} 8&\cellcolor{ForestGreen!60}10&	\cellcolor{blue!25}13&	15&	\cellcolor{ForestGreen!60}17&\cellcolor{blue!25}	20&	22&	\cellcolor{blue!25}24&\cellcolor{blue!25}	27&	29&	\cellcolor{blue!25}31&	\cellcolor{blue!25}34&36&	\cellcolor{blue!25}38\\
\hline
4&	\cellcolor{blue!25} 8&	\cellcolor{yellow!25}10&\cellcolor{ForestGreen!60}	13&\cellcolor{blue!25}	17&	\cellcolor{yellow!25}19&	\cellcolor{ForestGreen!60}22&\cellcolor{blue!25}	26&\cellcolor{yellow!25}	28&	\cellcolor{ForestGreen!60}31&\cellcolor{blue!25}	35&\cellcolor{yellow!25}	37&	\cellcolor{ForestGreen!60}40&	\cellcolor{blue!25}44&\cellcolor{yellow!25}	46&	\cellcolor{ForestGreen!60}49\\
\hline
5&	\cellcolor{blue!25} 10&	\cellcolor{yellow!25}13&	\cellcolor{ForestGreen!60}16&	\cellcolor{blue!25}21&	24&	\cellcolor{yellow!25}28&	\cellcolor{blue!25}32&\cellcolor{yellow!25}	35&	39&\cellcolor{ForestGreen!60}	43&	47&\cellcolor{yellow!25}	50&	\cellcolor{blue!25}54&\cellcolor{yellow!25}	58&	62\\
\hline
6&	\cellcolor{blue!25} 12&\cellcolor{yellow!25}	15&	\cellcolor{ForestGreen!60}19&	\cellcolor{blue!25}25&	\cellcolor{yellow!25}28&\cellcolor{yellow!25}	33&	\cellcolor{blue!25}38&	\cellcolor{yellow!25}42&\cellcolor{yellow!25}	46&	\cellcolor{ForestGreen!60}51&	\cellcolor{yellow!25}55&\cellcolor{yellow!25}	60&	\cellcolor{ForestGreen!60}64&\cellcolor{yellow!25}	69&	\cellcolor{yellow!25}73\\
\hline
7&	\cellcolor{blue!25} 14&	\cellcolor{yellow!25}18&	\cellcolor{ForestGreen!60}22&	\cellcolor{blue!25}29&	33&\cellcolor{yellow!25}	39&\cellcolor{blue!25}	44&	\cellcolor{yellow!25}49&	54&	\cellcolor{yellow!25}60&	65&\cellcolor{yellow!25}	70&	75&\cellcolor{yellow!25}	81&	86\\
\hline
8&	\cellcolor{blue!25} 16&	\cellcolor{yellow!25}20&	\cellcolor{ForestGreen!60}25&	\cellcolor{blue!25}33&	\cellcolor{yellow!25}37&\cellcolor{yellow!25}	44&	\cellcolor{blue!25}50&\cellcolor{yellow!25}	56&	\cellcolor{yellow!25}61&\cellcolor{yellow!25}	68&	\cellcolor{yellow!25}73&\cellcolor{yellow!25}	80&	\cellcolor{yellow!25}85&\cellcolor{yellow!25}	92&	\cellcolor{yellow!25}97\\
\hline
9&	\cellcolor{blue!25} 18&	\cellcolor{yellow!25}23&	\cellcolor{ForestGreen!60}28&	\cellcolor{blue!25}37&	42&\cellcolor{yellow!25}	50&	\cellcolor{blue!25}56&\cellcolor{yellow!25}	63&	69&\cellcolor{yellow!25}	77&	83&\cellcolor{yellow!25}	90&	96&\cellcolor{yellow!25}	104 & 110\\
\hline
10&	\cellcolor{blue!25} 20&\cellcolor{yellow!25}	25&\cellcolor{ForestGreen!60}	31&	\cellcolor{blue!25}41&\cellcolor{yellow!25}	46&\cellcolor{yellow!25}	55&	\cellcolor{blue!25}62&\cellcolor{yellow!25}	70&	\cellcolor{yellow!25}76&\cellcolor{yellow!25}	85&	\cellcolor{yellow!25}91&\cellcolor{yellow!25}	100 & \cellcolor{yellow!25}106 &\cellcolor{yellow!25} 115	&\cellcolor{yellow!25}121\\
\hline
11&	\cellcolor{blue!25} 22&	\cellcolor{yellow!25}28&	\cellcolor{ForestGreen!60}34&\cellcolor{blue!25}	45&	51&\cellcolor{yellow!25}	61&	\cellcolor{blue!25}68&\cellcolor{yellow!25}	77&	84&	\cellcolor{yellow!25}94&	101 &\cellcolor{yellow!25} 110 &	117	& \cellcolor{yellow!25}127 & 134\\
\hline
12&	\cellcolor{blue!25}24&\cellcolor{yellow!25}	30&	\cellcolor{ForestGreen!60}37&\cellcolor{blue!25}	49&\cellcolor{yellow!25}	55&	\cellcolor{yellow!25}66&	\cellcolor{blue!25}74&	\cellcolor{yellow!25}84&\cellcolor{yellow!25}	91&\cellcolor{yellow!25}	102 & \cellcolor{yellow!25}109 &\cellcolor{yellow!25}	120 &\cellcolor{yellow!25} 127 &\cellcolor{yellow!25}	138	& \cellcolor{yellow!25}145\\
\hline
13&	\cellcolor{blue!25} 26&	\cellcolor{yellow!25}33&	\cellcolor{ForestGreen!60}40&	\cellcolor{blue!25}53&	60&\cellcolor{yellow!25}	72&\cellcolor{blue!25}	80&\cellcolor{yellow!25}	91&	99&	\cellcolor{yellow!25}111 & 119 &\cellcolor{yellow!25}	130 & 138 &	\cellcolor{yellow!25}150	& 158\\
\hline
14&	\cellcolor{blue!25} 28&\cellcolor{yellow!25}	35&	\cellcolor{ForestGreen!60}43&	\cellcolor{blue!25}57&	\cellcolor{yellow!25}64&\cellcolor{yellow!25}	77&	\cellcolor{blue!25}86&\cellcolor{yellow!25}	98&	\cellcolor{yellow!25}106&\cellcolor{yellow!25} 119& \cellcolor{yellow!25}127& \cellcolor{yellow!25}140&\cellcolor{yellow!25}	148&\cellcolor{yellow!25} 161 &\cellcolor{yellow!25}169\\
\hline
15&	\cellcolor{blue!25} 30&	\cellcolor{yellow!25}38&	\cellcolor{ForestGreen!60}46&	\cellcolor{blue!25}61&	69&\cellcolor{yellow!25}	83&	\cellcolor{blue!25}92&\cellcolor{yellow!25}	105& 114&\cellcolor{yellow!25} 128&	137&\cellcolor{yellow!25} 150& 159&\cellcolor{yellow!25} 173	&182\\
\hline
16&	\cellcolor{blue!25} 32&\cellcolor{yellow!25}	40&	\cellcolor{ForestGreen!60}49&\cellcolor{blue!25}	65&	\cellcolor{yellow!25}73&\cellcolor{yellow!25}	88&\cellcolor{blue!25}	98&	\cellcolor{yellow!25}112& \cellcolor{yellow!25}121& \cellcolor{yellow!25}136&\cellcolor{yellow!25}	145& \cellcolor{yellow!25}160&\cellcolor{yellow!25}	169& \cellcolor{yellow!25}184 &\cellcolor{yellow!25}193 \\
\hline
\end{tabularx}
\begin{tabular}{ll}
\textcolor{yellow!25}{${\blacksquare}$}  Efficient configurations obtained by the brick pattern \\
\textcolor{blue!25}{${\blacksquare}$}   Efficient configurations obtained by the comb pattern \\
\textcolor{ForestGreen!60}{${\blacksquare}$}   Efficient configurations obtained by both the brick and the comb pattern
\end{tabular}
\caption{Table of occupancies of efficient configurations as calculated by IBM ILOG CPLEX.}\label{tablicica}
\end{table}

Note that the maximum occupancy is often attained by a number of different configurations. In some cases, the IP solutions resemble configurations with one of the special patterns introduced before, such as the brick or comb patterns, see Figure \ref{fig:representationexamples}.

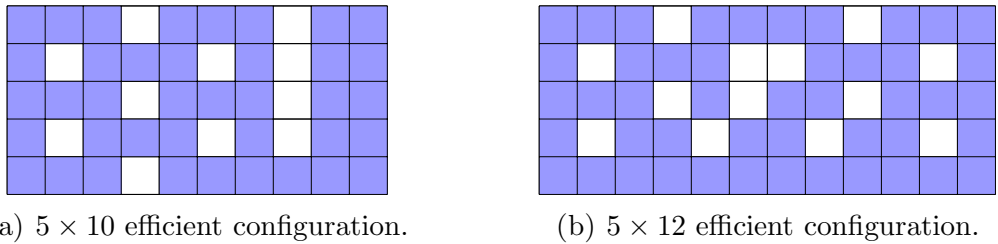
\begin{figure}
	\centering
	\begin{subfigure}{.5\textwidth}
		\centering
		\begin{tikzpicture}[scale = 0.5]
		\draw[step=1cm,black,very thin] (0, 0) grid (10,5);
		\fill[blue!40!white] (0,0) rectangle (1,1);
		\fill[blue!40!white] (0,1) rectangle (1,2);
		\fill[blue!40!white] (0,2) rectangle (1,3);
		\fill[blue!40!white] (0,3) rectangle (1,4);
		\fill[blue!40!white] (0,4) rectangle (1,5);
		
		\fill[blue!40!white] (1,0) rectangle (2,1);
		\fill[blue!40!white] (1,2) rectangle (2,3);
		\fill[blue!40!white] (1,4) rectangle (2,5);
		
		\fill[blue!40!white] (2,1) rectangle (3,2);
		\fill[blue!40!white] (2,0) rectangle (3,1);
		\fill[blue!40!white] (2,2) rectangle (3,3);
		\fill[blue!40!white] (2,3) rectangle (3,4);
		\fill[blue!40!white] (2,4) rectangle (3,5);
		
		\fill[blue!40!white] (3,1) rectangle (4,2);
		\fill[blue!40!white] (3,3) rectangle (4,4);
		
		\fill[blue!40!white] (4,0) rectangle (5,1);
		\fill[blue!40!white] (4,1) rectangle (5,2);
		\fill[blue!40!white] (4,2) rectangle (5,3);
		\fill[blue!40!white] (4,3) rectangle (5,4);
		\fill[blue!40!white] (4,4) rectangle (5,5);
		
		\fill[blue!40!white] (5,0) rectangle (6,1);
		\fill[blue!40!white] (5,2) rectangle (6,3);
		\fill[blue!40!white] (5,4) rectangle (6,5);
		
		\fill[blue!40!white] (6,0) rectangle (7,1);
		\fill[blue!40!white] (6,1) rectangle (7,2);
		\fill[blue!40!white] (6,2) rectangle (7,3);
		\fill[blue!40!white] (6,3) rectangle (7,4);
		\fill[blue!40!white] (6,4) rectangle (7,5);
		
		\fill[blue!40!white] (7,0) rectangle (8,1);
		
		\fill[blue!40!white] (8,0) rectangle (9,1);
		\fill[blue!40!white] (8,1) rectangle (9,2);
		\fill[blue!40!white] (8,2) rectangle (9,3);
		\fill[blue!40!white] (8,3) rectangle (9,4);
		\fill[blue!40!white] (8,4) rectangle (9,5);
		
		\fill[blue!40!white] (9,0) rectangle (10,1);
		\fill[blue!40!white] (9,1) rectangle (10,2);
		\fill[blue!40!white] (9,2) rectangle (10,3);
		\fill[blue!40!white] (9,3) rectangle (10,4);
		\fill[blue!40!white] (9,4) rectangle (10,5);

		\draw[step=1cm,black,very thin] (0, 0) grid (10,5);
		\end{tikzpicture}
		\caption{$5\times 10$ efficient configuration.}\label{fig:efficientBrickComb}
	\end{subfigure}%
	\begin{subfigure}{.5\textwidth}
		\centering
		\begin{tikzpicture}[scale = 0.5]
		\draw[step=1cm,black,very thin] (0, 0) grid (12,5);
		\fill[blue!40!white] (0,0) rectangle (1,1);
		\fill[blue!40!white] (0,1) rectangle (1,2);
		\fill[blue!40!white] (0,2) rectangle (1,3);
		\fill[blue!40!white] (0,3) rectangle (1,4);
		\fill[blue!40!white] (0,4) rectangle (1,5);
		
		\fill[blue!40!white] (1,0) rectangle (2,1);
		\fill[blue!40!white] (1,2) rectangle (2,3);
		\fill[blue!40!white] (1,4) rectangle (2,5);
		
		\fill[blue!40!white] (2,1) rectangle (3,2);
		\fill[blue!40!white] (2,0) rectangle (3,1);
		\fill[blue!40!white] (2,2) rectangle (3,3);
		\fill[blue!40!white] (2,3) rectangle (3,4);
		\fill[blue!40!white] (2,4) rectangle (3,5);
		
		\fill[blue!40!white] (3,0) rectangle (4,1);
		\fill[blue!40!white] (3,1) rectangle (4,2);
		\fill[blue!40!white] (3,3) rectangle (4,4);
		
		\fill[blue!40!white] (4,0) rectangle (5,1);
		\fill[blue!40!white] (4,2) rectangle (5,3);
		\fill[blue!40!white] (4,3) rectangle (5,4);
		\fill[blue!40!white] (4,4) rectangle (5,5);
		
		\fill[blue!40!white] (5,0) rectangle (6,1);
		\fill[blue!40!white] (5,1) rectangle (6,2);
		\fill[blue!40!white] (5,4) rectangle (6,5);
		
		\fill[blue!40!white] (6,0) rectangle (7,1);
		\fill[blue!40!white] (6,1) rectangle (7,2);
		\fill[blue!40!white] (6,2) rectangle (7,3);
		\fill[blue!40!white] (6,4) rectangle (7,5);
		
		\fill[blue!40!white] (7,0) rectangle (8,1);
		\fill[blue!40!white] (7,2) rectangle (8,3);
		\fill[blue!40!white] (7,3) rectangle (8,4);
		\fill[blue!40!white] (7,4) rectangle (8,5);
		
		\fill[blue!40!white] (8,0) rectangle (9,1);
		\fill[blue!40!white] (8,1) rectangle (9,2);
		\fill[blue!40!white] (8,3) rectangle (9,4);
		
		\fill[blue!40!white] (9,0) rectangle (10,1);
		\fill[blue!40!white] (9,1) rectangle (10,2);
		\fill[blue!40!white] (9,2) rectangle (10,3);
		\fill[blue!40!white] (9,3) rectangle (10,4);
		\fill[blue!40!white] (9,4) rectangle (10,5);
		
		\fill[blue!40!white] (10,0) rectangle (11,1);
		\fill[blue!40!white] (10,2) rectangle (11,3);
		\fill[blue!40!white] (10,4) rectangle (11,5);
		
		\fill[blue!40!white] (11,0) rectangle (12,1);
		\fill[blue!40!white] (11,1) rectangle (12,2);
		\fill[blue!40!white] (11,2) rectangle (12,3);
		\fill[blue!40!white] (11,3) rectangle (12,4);
		\fill[blue!40!white] (11,4) rectangle (12,5);

		\draw[step=1cm,black,very thin] (0, 0) grid (12,5);
		\end{tikzpicture}
		\caption{$5\times 12$ efficient configuration.}\label{fig:efficientNonSymmetric}
		\label{efficientexample2}
	\end{subfigure}%
	\caption{Examples of efficient configurations.}
\end{figure}

It is worth noting that a large portion of calculated occupancies of the efficient configurations could not be obtained by either brick nor comb patterns. Some of them are obtained by a certain combination of the two, see Figure \ref{fig:efficientBrickComb}. Some are obtained by the configurations with yet unexplained underlying patterns, see Figure \ref{fig:efficientNonSymmetric}.

\begin{figure}
\centering
\begin{subfigure}{.5\textwidth}
  \centering
  \begin{tikzpicture}[scale = 0.5]
            \draw[step=1cm,black,very thin] (0, 0) grid (14,14);
            \fill[blue!40!white] (0,0) rectangle (1,1);
            \fill[blue!40!white] (0,1) rectangle (1,2);
            \fill[blue!40!white] (0,2) rectangle (1,3);
            \fill[blue!40!white] (0,3) rectangle (1,4);
            \fill[blue!40!white] (0,4) rectangle (1,5);
            \fill[blue!40!white] (0,5) rectangle (1,6);
            \fill[blue!40!white] (0,6) rectangle (1,7);
            \fill[blue!40!white] (0,7) rectangle (1,8);
            \fill[blue!40!white] (0,8) rectangle (1,9);
            \fill[blue!40!white] (0,9) rectangle (1,10);
            \fill[blue!40!white] (0,10) rectangle (1,11);
            \fill[blue!40!white] (0,11) rectangle (1,12);
            \fill[blue!40!white] (0,12) rectangle (1,13);
            \fill[blue!40!white] (0,13) rectangle (1,14);

            \fill[blue!40!white] (1,0) rectangle (2,1);
            \fill[blue!40!white] (1,1) rectangle (2,2);
            \fill[blue!40!white] (1,2) rectangle (2,3);
            \fill[blue!40!white] (1,5) rectangle (2,6);
            \fill[blue!40!white] (1,7) rectangle (2,8);
            \fill[blue!40!white] (1,9) rectangle (2,10);
            \fill[blue!40!white] (1,11) rectangle (2,12);
            \fill[blue!40!white] (1,12) rectangle (2,13);
            \fill[blue!40!white] (1,13) rectangle (2,14);

            \fill[blue!40!white] (2,0) rectangle (3,1);
            \fill[blue!40!white] (2,3) rectangle (3,4);
            \fill[blue!40!white] (2,4) rectangle (3,5);
            \fill[blue!40!white] (2,5) rectangle (3,6);
            \fill[blue!40!white] (2,6) rectangle (3,7);
            \fill[blue!40!white] (2,7) rectangle (3,8);
            \fill[blue!40!white] (2,8) rectangle (3,9);
            \fill[blue!40!white] (2,9) rectangle (3,10);
            \fill[blue!40!white] (2,10) rectangle (3,11);
            \fill[blue!40!white] (2,11) rectangle (3,12);

            \fill[blue!40!white] (3,0) rectangle (4,1);
            \fill[blue!40!white] (3,1) rectangle (4,2);
            \fill[blue!40!white] (3,2) rectangle (4,3);
            \fill[blue!40!white] (3,4) rectangle (4,5);
            \fill[blue!40!white] (3,6) rectangle (4,7);
            \fill[blue!40!white] (3,8) rectangle (4,9);
            \fill[blue!40!white] (3,10) rectangle (4,11);
            \fill[blue!40!white] (3,12) rectangle (4,13);
            \fill[blue!40!white] (3,13) rectangle (4,14);

            \fill[blue!40!white] (4,0) rectangle (5,1);
            \fill[blue!40!white] (4,1) rectangle (5,2);
            \fill[blue!40!white] (4,2) rectangle (5,3);
            \fill[blue!40!white] (4,3) rectangle (5,4);
            \fill[blue!40!white] (4,4) rectangle (5,5);
            \fill[blue!40!white] (4,5) rectangle (5,6);
            \fill[blue!40!white] (4,6) rectangle (5,7);
            \fill[blue!40!white] (4,7) rectangle (5,8);
            \fill[blue!40!white] (4,8) rectangle (5,9);
            \fill[blue!40!white] (4,9) rectangle (5,10);
            \fill[blue!40!white] (4,10) rectangle (5,11);
            \fill[blue!40!white] (4,11) rectangle (5,12);
            \fill[blue!40!white] (4,13) rectangle (5,14);

            \fill[blue!40!white] (5,0) rectangle (6,1);
            \fill[blue!40!white] (5,3) rectangle (6,4);
            \fill[blue!40!white] (5,5) rectangle (6,6);
            \fill[blue!40!white] (5,7) rectangle (6,8);
            \fill[blue!40!white] (5,9) rectangle (6,10);
            \fill[blue!40!white] (5,11) rectangle (6,12);
            \fill[blue!40!white] (5,12) rectangle (6,13);
            \fill[blue!40!white] (5,13) rectangle (6,14);

            \fill[blue!40!white] (6,0) rectangle (7,1);
            \fill[blue!40!white] (6,1) rectangle (7,2);
            \fill[blue!40!white] (6,2) rectangle (7,3);
            \fill[blue!40!white] (6,3) rectangle (7,4);
            \fill[blue!40!white] (6,4) rectangle (7,5);
            \fill[blue!40!white] (6,5) rectangle (7,6);
            \fill[blue!40!white] (6,6) rectangle (7,7);
            \fill[blue!40!white] (6,7) rectangle (7,8);
            \fill[blue!40!white] (6,8) rectangle (7,9);
            \fill[blue!40!white] (6,9) rectangle (7,10);
            \fill[blue!40!white] (6,10) rectangle (7,11);
            \fill[blue!40!white] (6,12) rectangle (7,13);
			
            \fill[blue!40!white] (7,0) rectangle (8,1);
            \fill[blue!40!white] (7,1) rectangle (8,2);
            \fill[blue!40!white] (7,2) rectangle (8,3);
            \fill[blue!40!white] (7,4) rectangle (8,5);
            \fill[blue!40!white] (7,6) rectangle (8,7);
            \fill[blue!40!white] (7,8) rectangle (8,9);
            \fill[blue!40!white] (7,10) rectangle (8,11);
            \fill[blue!40!white] (7,11) rectangle (8,12);
            \fill[blue!40!white] (7,12) rectangle (8,13);
            \fill[blue!40!white] (7,13) rectangle (8,14);

            \fill[blue!40!white] (8,0) rectangle (9,1);
            \fill[blue!40!white] (8,3) rectangle (9,4);
            \fill[blue!40!white] (8,4) rectangle (9,5);
            \fill[blue!40!white] (8,5) rectangle (9,6);
            \fill[blue!40!white] (8,6) rectangle (9,7);
            \fill[blue!40!white] (8,7) rectangle (9,8);
            \fill[blue!40!white] (8,8) rectangle (9,9);
            \fill[blue!40!white] (8,9) rectangle (9,10);
            \fill[blue!40!white] (8,11) rectangle (9,12);
            \fill[blue!40!white] (8,13) rectangle (9,14);

            \fill[blue!40!white] (9,0) rectangle (10,1);
            \fill[blue!40!white] (9,1) rectangle (10,2);
            \fill[blue!40!white] (9,3) rectangle (10,4);
            \fill[blue!40!white] (9,5) rectangle (10,6);
            \fill[blue!40!white] (9,7) rectangle (10,8);
            \fill[blue!40!white] (9,9) rectangle (10,10);
            \fill[blue!40!white] (9,10) rectangle (10,11);
            \fill[blue!40!white] (9,11) rectangle (10,12);
            \fill[blue!40!white] (9,12) rectangle (10,13);
            \fill[blue!40!white] (9,13) rectangle (10,14);

            \fill[blue!40!white] (10,0) rectangle (11,1);
            \fill[blue!40!white] (10,1) rectangle (11,2);
            \fill[blue!40!white] (10,2) rectangle (11,3);
            \fill[blue!40!white] (10,3) rectangle (11,4);
            \fill[blue!40!white] (10,4) rectangle (11,5);
            \fill[blue!40!white] (10,5) rectangle (11,6);
            \fill[blue!40!white] (10,6) rectangle (11,7);
            \fill[blue!40!white] (10,7) rectangle (11,8);
            \fill[blue!40!white] (10,8) rectangle (11,9);
            \fill[blue!40!white] (10,10) rectangle (11,11);
            \fill[blue!40!white] (10,12) rectangle (11,13);

            \fill[blue!40!white] (11,0) rectangle (12,1);
            \fill[blue!40!white] (11,4) rectangle (12,5);
            \fill[blue!40!white] (11,6) rectangle (12,7);
            \fill[blue!40!white] (11,8) rectangle (12,9);
            \fill[blue!40!white] (11,9) rectangle (12,10);
            \fill[blue!40!white] (11,10) rectangle (12,11);
            \fill[blue!40!white] (11,11) rectangle (12,12);
            \fill[blue!40!white] (11,12) rectangle (12,13);
            \fill[blue!40!white] (11,13) rectangle (12,14);

            \fill[blue!40!white] (12,0) rectangle (13,1);
            \fill[blue!40!white] (12,1) rectangle (13,2);
            \fill[blue!40!white] (12,2) rectangle (13,3);
            \fill[blue!40!white] (12,4) rectangle (13,5);
            \fill[blue!40!white] (12,6) rectangle (13,7);
            \fill[blue!40!white] (12,7) rectangle (13,8);
            \fill[blue!40!white] (12,9) rectangle (13,10);
            \fill[blue!40!white] (12,11) rectangle (13,12);
            \fill[blue!40!white] (12,13) rectangle (13,14);
            
            \fill[blue!40!white] (13,0) rectangle (14,1);
            \fill[blue!40!white] (13,1) rectangle (14,2);
            \fill[blue!40!white] (13,2) rectangle (14,3);
            \fill[blue!40!white] (13,3) rectangle (14,4);
            \fill[blue!40!white] (13,4) rectangle (14,5);
            \fill[blue!40!white] (13,5) rectangle (14,6);
            \fill[blue!40!white] (13,6) rectangle (14,7);
            \fill[blue!40!white] (13,7) rectangle (14,8);
            \fill[blue!40!white] (13,8) rectangle (14,9);
            \fill[blue!40!white] (13,9) rectangle (14,10);
            \fill[blue!40!white] (13,10) rectangle (14,11);
            \fill[blue!40!white] (13,11) rectangle (14,12);
            \fill[blue!40!white] (13,12) rectangle (14,13);
            \fill[blue!40!white] (13,13) rectangle (14,14);

            \draw[step=1cm,black,very thin] (0,0) grid (14,14);
        \end{tikzpicture}
  \caption{IP solution.}
\end{subfigure}%
\begin{subfigure}{.5\textwidth}
  \centering
  \begin{tikzpicture}[scale = 0.5]
            \draw[step=1cm,black,very thin] (0, 0) grid (14,14);
            \fill[blue!40!white] (0,0) rectangle (1,1);
            \fill[blue!40!white] (0,1) rectangle (1,2);
            \fill[blue!40!white] (0,2) rectangle (1,3);
            \fill[blue!40!white] (0,3) rectangle (1,4);
            \fill[blue!40!white] (0,4) rectangle (1,5);
            \fill[blue!40!white] (0,5) rectangle (1,6);
            \fill[blue!40!white] (0,6) rectangle (1,7);
            \fill[blue!40!white] (0,7) rectangle (1,8);
            \fill[blue!40!white] (0,8) rectangle (1,9);
            \fill[blue!40!white] (0,9) rectangle (1,10);
            \fill[blue!40!white] (0,10) rectangle (1,11);
            \fill[blue!40!white] (0,11) rectangle (1,12);
            \fill[blue!40!white] (0,12) rectangle (1,13);
            \fill[blue!40!white] (0,13) rectangle (1,14);

            \fill[blue!40!white] (1,1) rectangle (2,2);
   
            \fill[blue!40!white] (1,3) rectangle (2,4);
      
            \fill[blue!40!white] (1,5) rectangle (2,6);
      
            \fill[blue!40!white] (1,7) rectangle (2,8);
         
            \fill[blue!40!white] (1,9) rectangle (2,10);
     
            \fill[blue!40!white] (1,11) rectangle (2,12);
        
            \fill[blue!40!white] (1,13) rectangle (2,14);

            \fill[blue!40!white] (2,0) rectangle (3,1);
            \fill[blue!40!white] (2,1) rectangle (3,2);
            \fill[blue!40!white] (2,2) rectangle (3,3);
            \fill[blue!40!white] (2,3) rectangle (3,4);
            \fill[blue!40!white] (2,4) rectangle (3,5);
            \fill[blue!40!white] (2,5) rectangle (3,6);
            \fill[blue!40!white] (2,6) rectangle (3,7);
            \fill[blue!40!white] (2,7) rectangle (3,8);
            \fill[blue!40!white] (2,8) rectangle (3,9);
            \fill[blue!40!white] (2,9) rectangle (3,10);
            \fill[blue!40!white] (2,10) rectangle (3,11);
            \fill[blue!40!white] (2,11) rectangle (3,12);
            \fill[blue!40!white] (2,12) rectangle (3,13);
            \fill[blue!40!white] (2,13) rectangle (3,14);

            \fill[blue!40!white] (3,0) rectangle (4,1);
     
            \fill[blue!40!white] (3,2) rectangle (4,3);
     
            \fill[blue!40!white] (3,4) rectangle (4,5);
        
            \fill[blue!40!white] (3,6) rectangle (4,7);
        
            \fill[blue!40!white] (3,8) rectangle (4,9);
       
            \fill[blue!40!white] (3,10) rectangle (4,11);
        
            \fill[blue!40!white] (3,12) rectangle (4,13);

            \fill[blue!40!white] (4,0) rectangle (5,1);
            \fill[blue!40!white] (4,1) rectangle (5,2);
            \fill[blue!40!white] (4,2) rectangle (5,3);
            \fill[blue!40!white] (4,3) rectangle (5,4);
            \fill[blue!40!white] (4,4) rectangle (5,5);
            \fill[blue!40!white] (4,5) rectangle (5,6);
            \fill[blue!40!white] (4,6) rectangle (5,7);
            \fill[blue!40!white] (4,7) rectangle (5,8);
            \fill[blue!40!white] (4,8) rectangle (5,9);
            \fill[blue!40!white] (4,9) rectangle (5,10);
            \fill[blue!40!white] (4,10) rectangle (5,11);
            \fill[blue!40!white] (4,11) rectangle (5,12);
            \fill[blue!40!white] (4,12) rectangle (5,13);
            \fill[blue!40!white] (4,13) rectangle (5,14);

            \fill[blue!40!white] (5,1) rectangle (6,2);
         
            \fill[blue!40!white] (5,3) rectangle (6,4);
        
            \fill[blue!40!white] (5,5) rectangle (6,6);
        
            \fill[blue!40!white] (5,7) rectangle (6,8);
        
            \fill[blue!40!white] (5,9) rectangle (6,10);
       
            \fill[blue!40!white] (5,11) rectangle (6,12);
        
            \fill[blue!40!white] (5,13) rectangle (6,14);

            \fill[blue!40!white] (6,0) rectangle (7,1);
            \fill[blue!40!white] (6,1) rectangle (7,2);
            \fill[blue!40!white] (6,2) rectangle (7,3);
            \fill[blue!40!white] (6,3) rectangle (7,4);
            \fill[blue!40!white] (6,4) rectangle (7,5);
            \fill[blue!40!white] (6,5) rectangle (7,6);
            \fill[blue!40!white] (6,6) rectangle (7,7);
            \fill[blue!40!white] (6,7) rectangle (7,8);
            \fill[blue!40!white] (6,8) rectangle (7,9);
            \fill[blue!40!white] (6,9) rectangle (7,10);
            \fill[blue!40!white] (6,10) rectangle (7,11);
            \fill[blue!40!white] (6,11) rectangle (7,12);
            \fill[blue!40!white] (6,12) rectangle (7,13);
            \fill[blue!40!white] (6,13) rectangle (7,14);

            \fill[blue!40!white] (7,0) rectangle (8,1);
           
            \fill[blue!40!white] (7,2) rectangle (8,3);
            
            \fill[blue!40!white] (7,4) rectangle (8,5);
            
            \fill[blue!40!white] (7,6) rectangle (8,7);
            
            \fill[blue!40!white] (7,8) rectangle (8,9);
            
            \fill[blue!40!white] (7,10) rectangle (8,11);
            
            \fill[blue!40!white] (7,12) rectangle (8,13);

            \fill[blue!40!white] (8,0) rectangle (9,1);
            \fill[blue!40!white] (8,1) rectangle (9,2);
            \fill[blue!40!white] (8,2) rectangle (9,3);
            \fill[blue!40!white] (8,3) rectangle (9,4);
            \fill[blue!40!white] (8,4) rectangle (9,5);
            \fill[blue!40!white] (8,5) rectangle (9,6);
            \fill[blue!40!white] (8,6) rectangle (9,7);
            \fill[blue!40!white] (8,7) rectangle (9,8);
            \fill[blue!40!white] (8,8) rectangle (9,9);
            \fill[blue!40!white] (8,9) rectangle (9,10);
            \fill[blue!40!white] (8,10) rectangle (9,11);
            \fill[blue!40!white] (8,11) rectangle (9,12);
            \fill[blue!40!white] (8,12) rectangle (9,13);
            \fill[blue!40!white] (8,13) rectangle (9,14);

            \fill[blue!40!white] (9,1) rectangle (10,2);
            
            \fill[blue!40!white] (9,3) rectangle (10,4);
            
            \fill[blue!40!white] (9,5) rectangle (10,6);
           
            \fill[blue!40!white] (9,7) rectangle (10,8);
            
            \fill[blue!40!white] (9,9) rectangle (10,10);
           
            \fill[blue!40!white] (9,11) rectangle (10,12);
            
            \fill[blue!40!white] (9,13) rectangle (10,14);

            \fill[blue!40!white] (10,0) rectangle (11,1);
            \fill[blue!40!white] (10,1) rectangle (11,2);
            \fill[blue!40!white] (10,2) rectangle (11,3);
            \fill[blue!40!white] (10,3) rectangle (11,4);
            \fill[blue!40!white] (10,4) rectangle (11,5);
            \fill[blue!40!white] (10,5) rectangle (11,6);
            \fill[blue!40!white] (10,6) rectangle (11,7);
            \fill[blue!40!white] (10,7) rectangle (11,8);
            \fill[blue!40!white] (10,8) rectangle (11,9);
            \fill[blue!40!white] (10,9) rectangle (11,10);
            \fill[blue!40!white] (10,10) rectangle (11,11);
            \fill[blue!40!white] (10,11) rectangle (11,12);
            \fill[blue!40!white] (10,12) rectangle (11,13);
            \fill[blue!40!white] (10,13) rectangle (11,14);
            
            \fill[blue!40!white] (11,0) rectangle (12,1);
            
            \fill[blue!40!white] (11,2) rectangle (12,3);
            
            \fill[blue!40!white] (11,4) rectangle (12,5);
            
            \fill[blue!40!white] (11,6) rectangle (12,7);
            
            \fill[blue!40!white] (11,8) rectangle (12,9);
            
            \fill[blue!40!white] (11,10) rectangle (12,11);
            
            \fill[blue!40!white] (11,12) rectangle (12,13);

            \fill[blue!40!white] (12,0) rectangle (13,1);
            \fill[blue!40!white] (12,1) rectangle (13,2);
            \fill[blue!40!white] (12,2) rectangle (13,3);
            \fill[blue!40!white] (12,3) rectangle (13,4);
            \fill[blue!40!white] (12,4) rectangle (13,5);
            \fill[blue!40!white] (12,5) rectangle (13,6);
            \fill[blue!40!white] (12,6) rectangle (13,7);
            \fill[blue!40!white] (12,7) rectangle (13,8);
            \fill[blue!40!white] (12,8) rectangle (13,9);
            \fill[blue!40!white] (12,9) rectangle (13,10);
            \fill[blue!40!white] (12,10) rectangle (13,11);
            \fill[blue!40!white] (12,11) rectangle (13,12);
            \fill[blue!40!white] (12,12) rectangle (13,13);
            \fill[blue!40!white] (12,13) rectangle (13,14);
            
            \fill[blue!40!white] (13,0) rectangle (14,1);
            \fill[blue!40!white] (13,1) rectangle (14,2);
            
            \fill[blue!40!white] (13,3) rectangle (14,4);
            
            \fill[blue!40!white] (13,5) rectangle (14,6);
            
            \fill[blue!40!white] (13,7) rectangle (14,8);
            
            \fill[blue!40!white] (13,9) rectangle (14,10);
            
            \fill[blue!40!white] (13,11) rectangle (14,12);
            
            \fill[blue!40!white] (13,13) rectangle (14,14);

            \draw[step=1cm,black,very thin] (0,0) grid (14,14);
        \end{tikzpicture}
 		\caption{Brick pattern configuration.}
\end{subfigure}
\caption{The IP solution to the problem of finding efficient configurations and the corresponding brick pattern configuration on $14\times 14$ grid. Note that both have the maximum occupancy of $148$.}\label{fig:representationexamples}
\end{figure}
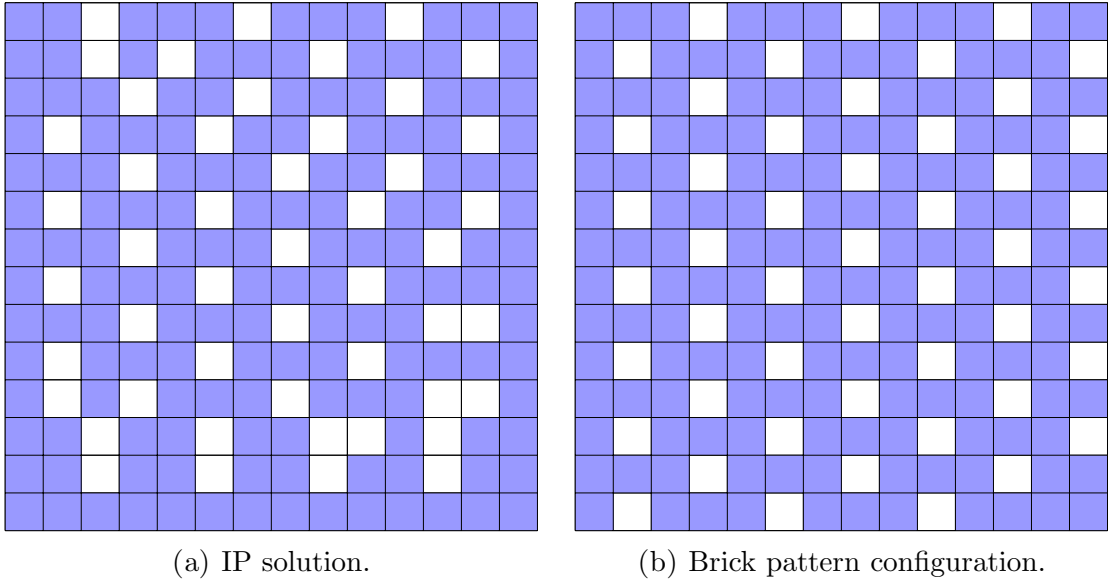

We leave the following question open.
\begin{question}\label{quest:Emn}
	What is the explicit expression for $E_{m,n}$?
\end{question}

\section{Alternative formulations of the problem}\label{sec:alt_formulations}

In this section we give alternative formulations  for the problem of finding efficient configurations in the hope of getting more people interested in answering Question \ref{quest:Emn}.

\subsection{Tilings}
The problem of finding \emph{efficient} configurations can be formulated as a tiling problem with overlaps and protrusions allowed; and rotations forbidden --- as follows. Let $C$ be a maximal configuration on an $m\times n$ grid. Now cover each empty lot by a $\bot$-tetromino which covers that lot as well as the neighboring lots to the east, west and north. Note that some tetrominoes may protrude from the grid, and some may overlap. The permissibility of $C$ guarantees that the $(m-1)\times (n-2)$ subgrid $\{1,2\dots,m-1\} \times \{2,3,\dots,n-1\}$ is completely covered by (at least one) $\bot$-tetromino. Conversely, each such a tiling guarantees the permissibility of the corresponding configuration. However, such a tiling does not guarantee that the corresponding configuration is maximal, even in the case when the tiling itself is maximal\footnote{The tiling is maximal if removing any tile leaves some lot in $\{1,2\dots,m-1\} \times \{2,3,\dots,n-1\}$ uncovered.}. Nevertheless, efficient configurations correspond exactly to tilings with the fewest number of tiles. Therefore $E_{m,n} = mn-k$ where $k$ is the number of tiles in an optimal tiling of the $(m-1)\times (n-2)$ grid.

This formulation of the tiling problem does not seem to be very common in literature. For some classical results about $T$-tetromino tilings see \cite{KP04,Walkup65,Criel08,steurer2009tilings}.

\subsection{Forbidden induced subgraph problem}

A tract of land can be represented as a rectangular lattice graph where each lot is represented by a vertex. In order to distinguish the north from the south we orient the edges as in Figure \ref{fig:latticegraph}. The problem of finding efficient configurations is equivalent to finding the largest set of vertices for which the induced graph does not contain either of the subgraphs in Figure \ref{fig:forbidden_subgraph}. The literature on the problem of forbidden induced subgraphs is extensive, see \cite{zaslavsky17,wiki-FGC}.

\begin{figure}
	\begin{tikzpicture}[darkstyle/.style={circle,draw,fill=gray!40}]
	\foreach \x in {0,...,5}
	\foreach \y in {0,...,3} 
	{\node [darkstyle] (\x\y) at (\x,\y) {};} 
	
	\foreach \x in {0,2,4}
	\foreach \y [count=\yi] in {0,...,2}{
		\draw[<-, line width = 1pt] (\x\y)--(\x\yi);
	}
	
	\foreach \x in {1,3,5}
	\foreach \y [count=\yi] in {0,...,2}{
		\draw[<-, line width = 1pt] (\x\yi)--(\x\y);
	}
	
	\draw[->, line width = 1pt] (00) -- (10);
	\draw[->, line width = 1pt] (01) -- (11);
	\draw[->, line width = 1pt] (02) -- (12);
	\draw[->, line width = 1pt] (03) -- (13);
	
	\draw[->, line width = 1pt] (20) -- (10);
	\draw[->, line width = 1pt] (21) -- (11);
	\draw[->, line width = 1pt] (22) -- (12);
	\draw[->, line width = 1pt] (23) -- (13);
	
	\draw[->, line width = 1pt] (20) -- (30);
	\draw[->, line width = 1pt] (21) -- (31);
	\draw[->, line width = 1pt] (22) -- (32);
	\draw[->, line width = 1pt] (23) -- (33);
	
	\draw[->, line width = 1pt] (40) -- (30);
	\draw[->, line width = 1pt] (41) -- (31);
	\draw[->, line width = 1pt] (42) -- (32);
	\draw[->, line width = 1pt] (43) -- (33);
	
	\draw[->, line width = 1pt] (40) -- (50);
	\draw[->, line width = 1pt] (41) -- (51);
	\draw[->, line width = 1pt] (42) -- (52);
	\draw[->, line width = 1pt] (43) -- (53);
	\end{tikzpicture}
	\caption{The graph representation of the $4\times 6$ tract of land.}\label{fig:latticegraph}
\end{figure}
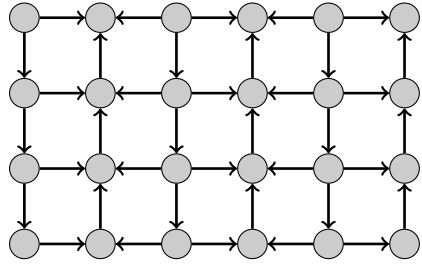

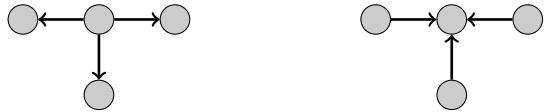
\begin{figure}
	\begin{subfigure}{.3\textwidth}\centering
		\begin{tikzpicture}[main/.style = {draw, circle,fill=gray!40}] 
		\node[main] (1) {}; 
		\node[main] (2) [right of=1] {};
		\node[main] (3) [right of=2] {}; 
		\node[main] (4) [below of=2] {};
		
		\draw[->, line width = 1pt] (2) -- (1);
		\draw[->, line width = 1pt] (2) -- (3);
		\draw[->, line width = 1pt] (2) -- (4);
		
		\end{tikzpicture}
	\end{subfigure}
	\begin{subfigure}{.3\textwidth}\centering
		\begin{tikzpicture}[main/.style = {draw, circle,fill=gray!40}] 
		\node[main] (1) {}; 
		\node[main] (2) [right of=1] {};
		\node[main] (3) [right of=2] {}; 
		\node[main] (4) [below of=2] {};
		
		\draw[->, line width = 1pt] (1) -- (2);
		\draw[->, line width = 1pt] (3) -- (2);
		\draw[->, line width = 1pt] (4) -- (2);
		
		\end{tikzpicture}
	\end{subfigure}
	\caption{The forbidden subgraphs.}\label{fig:forbidden_subgraph}
\end{figure}

\subsection{Subshift on \texorpdfstring{$\bbZ^2$}{Z2}}

Note that permissible configurations on $\bbZ^2$ can be interpreted as elements of the $\bbZ^2$-shift of finite type with the alphabet $\mathcal{A} = \{0,1\}$ where the set of forbidden patterns consists of a single pattern in Figure \ref{fig:forbidden_pattern}, see \cite{lemp2017shift}. An empty lot is represented by $0$, and occupied by $1$. It is interesting to note that the question whether the shift space defined by a set of forbidden patterns is empty, in general, undecidable \cite{Berger66, Robinson71}.

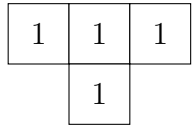
\begin{figure}
	\begin{tikzpicture}[scale = 0.8]
	
	\filldraw [fill=white, draw=black] (0,0) rectangle (1,1);
	\filldraw [fill=white, draw=black] (1,0) rectangle (2,1);
	\filldraw [fill=white, draw=black] (2,0) rectangle (3,1);
	\filldraw [fill=white, draw=black] (1,-1) rectangle (2,0);
	
	\node[] at (0.5,0.5) {$1$};
	\node[] at (1.5,0.5) {$1$};
	\node[] at (2.5,0.5) {$1$};
	\node[] at (1.5,-0.5) {$1$};
	\end{tikzpicture}
	\caption{The forbidden pattern.}\label{fig:forbidden_pattern}
\end{figure}

\section*{Acknowledgments} 

We wish to thank Juraj Bo\v{z}i\'{c} who introduced us to this problem that he came up with during his studies at Faculty of Architecture, University of Zagreb.

We additionally want to thank our colleagues Petar Baki\'{c}, Matija Ba\v{s}i\'{c} and Stipe Vidak thank to whom one particular instance of this problem ended up in the 10\textsuperscript{th} Middle European Mathematical Olympiad in V\"{o}klabruck, Austria (see \cite{MEMO}).

We also wish to thank Professor Tomislav Do\v{s}li\'{c} for fruitful and stimulating discussions.

\bibliographystyle{plain}
\bibliography{literature}

\end{document}